\RequirePackage[leqno]{amsmath}
\documentclass[leqno,a4paper]{amsart}
\usepackage[usenames,dvipsnames]{xcolor}
\usepackage{amsfonts}
\usepackage{enumitem}
\setlist[itemize]{topsep=0pt,after=\vspace{1.5\baselineskip}}
 \usepackage[colorlinks=true]{hyperref}
\usepackage{empheq}
\usepackage[T1]{fontenc}
\usepackage[utf8]{inputenc}
\usepackage{todonotes}

\def\R{\mathbb R}  
\newtheorem{theorem}{Theorem}[section]

\newtheorem{lemma}[theorem]{Lemma}

\newtheorem{remark}[theorem]{Remark}

\newcommand{\na}{\nabla}
\newcommand{\Om}{\Omega}
\newcommand{\Ombar}{\overline{\Omega}}
\newcommand{\Cdom}{C_{\partial\Om}}
\newcommand{\f}[2]{\frac{#1}{#2}}
\newcommand{\io}{\int_\Omega}
\newcommand{\intdom}{\int_{\partial\Om}}
\newcommand{\norm}[2][]{\left\|#2\right\|_{#1}}
\newcommand{\set}[1]{\{#1\}}
\newcommand{\kl}[1]{\left(#1\right)}
\newcommand{\calG}{\mathcal{G}}
\newcommand{\Tmax}{T_{max}}
\newcommand{\Lom}[1]{L^{#1}(\Om)}
\newcommand{\nn}{\nonumber}

\usepackage{newunicodechar}
\newunicodechar{∞}{\infty}
\newunicodechar{α}{\alpha}
\newunicodechar{β}{\beta}
\newunicodechar{γ}{\gamma}
\newunicodechar{δ}{\delta}
\newunicodechar{ε}{\varepsilon}
\newunicodechar{κ}{\kappa}
\newunicodechar{χ}{\chi}
\newunicodechar{μ}{\mu}
\newunicodechar{σ}{\sigma}
\newunicodechar{τ}{\tau}
\newunicodechar{ξ}{\xi}
\newunicodechar{φ}{\varphi}
\newunicodechar{Δ}{\Delta}
\newunicodechar{ℝ}{\mathbb{R}}

\title[Properties of solutions to a singular chemotaxis-consumption model] %Use the shortened version of the full title
      {Global existence and boundedness of solutions to a chemotaxis-consumption model with singular sensitivity}
\author[J. Lankeit and G. Viglialoro]{}
\keywords{Nonlinear parabolic systems, chemotaxis,  singular sensitivity, global existence, boundedness.\\
\textit{$^*$Corresponding author}: jlankeit@math.uni-paderborn.de}

\RequirePackage[normalem]{ulem} %DIF PREAMBLE
\RequirePackage{color}\definecolor{RED}{rgb}{1,0,0}\definecolor{BLUE}{rgb}{0,0,1} %DIF PREAMBLE
\providecommand{\DIFaddbegin}{} %DIF PREAMBLE
\providecommand{\DIFaddend}{} %DIF PREAMBLE
\providecommand{\DIFdelbegin}{} %DIF PREAMBLE
\providecommand{\DIFdelend}{} %DIF PREAMBLE
\providecommand{\DIFaddbeginFL}{} %DIF PREAMBLE
\providecommand{\DIFaddendFL}{} %DIF PREAMBLE
\providecommand{\DIFdelbeginFL}{} %DIF PREAMBLE
\providecommand{\DIFdelendFL}{} %DIF PREAMBLE
\newcommand{\DIFscaledelfig}{0.5}
\RequirePackage{settobox} %DIF PREAMBLE
\RequirePackage{letltxmacro} %DIF PREAMBLE
\newsavebox{\DIFdelgraphicsbox} %DIF PREAMBLE
\newlength{\DIFdelgraphicswidth} %DIF PREAMBLE
\newlength{\DIFdelgraphicsheight} %DIF PREAMBLE
\LetLtxMacro{\DIFOincludegraphics}{\includegraphics} %DIF PREAMBLE
\newcommand{\DIFaddincludegraphics}[2][]{{\color{blue}\fbox{\DIFOincludegraphics[#1]{#2}}}} %DIF PREAMBLE
\newcommand{\DIFdelincludegraphics}[2][]{% %DIF PREAMBLE
\sbox{\DIFdelgraphicsbox}{\DIFOincludegraphics[#1]{#2}}% %DIF PREAMBLE
\settoboxwidth{\DIFdelgraphicswidth}{\DIFdelgraphicsbox} %DIF PREAMBLE
\settoboxtotalheight{\DIFdelgraphicsheight}{\DIFdelgraphicsbox} %DIF PREAMBLE
\scalebox{\DIFscaledelfig}{% %DIF PREAMBLE
\parbox[b]{\DIFdelgraphicswidth}{\usebox{\DIFdelgraphicsbox}\\[-\baselineskip] \rule{\DIFdelgraphicswidth}{0em}}\llap{\resizebox{\DIFdelgraphicswidth}{\DIFdelgraphicsheight}{% %DIF PREAMBLE
\setlength{\unitlength}{\DIFdelgraphicswidth}% %DIF PREAMBLE
\begin{picture}(1,1)% %DIF PREAMBLE
\thicklines\linethickness{2pt} %DIF PREAMBLE
{\color[rgb]{1,0,0}\put(0,0){\framebox(1,1){}}}% %DIF PREAMBLE
{\color[rgb]{1,0,0}\put(0,0){\line( 1,1){1}}}% %DIF PREAMBLE
{\color[rgb]{1,0,0}\put(0,1){\line(1,-1){1}}}% %DIF PREAMBLE
\end{picture}% %DIF PREAMBLE
}\hspace*{3pt}}} %DIF PREAMBLE
} %DIF PREAMBLE
\LetLtxMacro{\DIFOaddbegin}{\DIFaddbegin} %DIF PREAMBLE
\LetLtxMacro{\DIFOaddend}{\DIFaddend} %DIF PREAMBLE
\LetLtxMacro{\DIFOdelbegin}{\DIFdelbegin} %DIF PREAMBLE
\LetLtxMacro{\DIFOdelend}{\DIFdelend} %DIF PREAMBLE
\DeclareRobustCommand{\DIFaddbegin}{\DIFOaddbegin \let\includegraphics\DIFaddincludegraphics} %DIF PREAMBLE
\DeclareRobustCommand{\DIFaddend}{\DIFOaddend \let\includegraphics\DIFOincludegraphics} %DIF PREAMBLE
\DeclareRobustCommand{\DIFdelbegin}{\DIFOdelbegin \let\includegraphics\DIFdelincludegraphics} %DIF PREAMBLE
\DeclareRobustCommand{\DIFdelend}{\DIFOaddend \let\includegraphics\DIFOincludegraphics} %DIF PREAMBLE
\LetLtxMacro{\DIFOaddbeginFL}{\DIFaddbeginFL} %DIF PREAMBLE
\LetLtxMacro{\DIFOaddendFL}{\DIFaddendFL} %DIF PREAMBLE
\LetLtxMacro{\DIFOdelbeginFL}{\DIFdelbeginFL} %DIF PREAMBLE
\LetLtxMacro{\DIFOdelendFL}{\DIFdelendFL} %DIF PREAMBLE
\DeclareRobustCommand{\DIFaddbeginFL}{\DIFOaddbeginFL \let\includegraphics\DIFaddincludegraphics} %DIF PREAMBLE
\DeclareRobustCommand{\DIFaddendFL}{\DIFOaddendFL \let\includegraphics\DIFOincludegraphics} %DIF PREAMBLE
\DeclareRobustCommand{\DIFdelbeginFL}{\DIFOdelbeginFL \let\includegraphics\DIFdelincludegraphics} %DIF PREAMBLE
\DeclareRobustCommand{\DIFdelendFL}{\DIFOaddendFL \let\includegraphics\DIFOincludegraphics} %DIF PREAMBLE
\begin{document}
\maketitle
\DIFdelbegin %DIFDELCMD < \maketitle
%DIFDELCMD < %%%
\DIFdelend 
%\maketitle

\centerline{\scshape Johannes Lankeit$^{1,*}$ \and  Giuseppe Viglialoro$^{2}$}
\medskip
{
 \footnotesize
 \centerline{$^1$Institut für Mathematik}
 \centerline{Universität  Paderborn }
 \centerline{Warburger Str. 100,  33098  Paderborn (Germany)}
 \medskip
 \centerline{$^2$Dipartimento di Matematica e Informatica}
 \centerline{Universit\`{a} di Cagliari}
 \centerline{V. le Merello 92, 09123. Cagliari (Italy)}
 \medskip
}

\bigskip
\begin{abstract}
In this paper we study the zero-flux chemotaxis-system 
\begin{equation*}
\begin{cases}
u_t=\Delta u -\chi \nabla \cdot (\frac{u}{v} \nabla v) \\
v_t=\Delta v-f(u)v 
\end{cases}
\end{equation*}
in a smooth and bounded domain $\Omega$ of $\mathbb{R}^2$,  with $\chi>0$ and $f\in C^1(\mathbb{R})$ essentially behaving like $u^\beta$, $0<\beta<1$. Precisely for $\chi<1$ and any sufficiently regular initial data $u(x,0)\geq 0$ and $v(x,0)>0$ on $\bar{\Omega}$, we show the existence of global classical solutions. Moreover, if additionally $m:=\int_\Omega u(x,0)$ is sufficiently small, then also their boundedness is achieved.
%  \textbf{MSC (2010):} 35Q92; 35A01; 35K51; 92C17\\
% \textbf{Keywords:} chemotaxis; classical solution; singular sensitivity; signal consumption; global existence; boundedness
%
%
\end{abstract}
% % % % % % % % % % %
%\tableofcontents 
% % % % % % % % % % % %
\section{Introduction and motivations}\label{IntroductionSection} 
Chemotaxis systems in the form of the classical Keller--Segel system (\cite{K-S-1970,horstmann,BellomoEtAl}) model aggregation phenomena in situations where cells are attracted by a signal they themselves emit. If they, instead, direct their movement in response to a substance they consume, the equation governing evolution of the signal concentration becomes much more amenable to providing uniform bounds on this concentration (although, in the most commonly used form, the derivation of bounds for its gradient is more negatively affected by a nonlinearity). Such systems have extensively been studied throughout the past few years, especially in the context of chemotaxis--fluid models, and the interested reader can find pointers to the rich literature for example in the introduction of \cite{caolankeit}.  

However, a new difficulty arises if such consumptive chemotaxis models incorporate the effect that small changes in a stimulus affect the response of a biological agent more heavily at a low signal level than the same changes would in presence of high signal concentrations (the so-called 'Weber-Fechner law of stimulus perception') in the way that the chemotactic sensitivity function is chosen singular, as in  
\begin{equation}\label{intro:thesystem}
 \begin{cases}
  u_t=\Delta u - \chi \nabla \cdot \left(\frac uv\nabla v\right)\\
  v_t=\Delta v - uv,
 \end{cases}
\end{equation}
where small signal concentrations enhance the possibly destabilizing cross-diffusive contribution of the chemotaxis term in the first equation. 

This system goes back to Keller and Segel studying the formation of travelling bands of \textit{E. coli} \cite{kellersegel_trav}. (For more results on travelling  wave solutions in this and related models, see \cite{wang_TWSsurvey}.) With respect to global existence of solutions it has been less extensively studied than 
its signal-production relative 
\begin{equation}\label{intro:productionsystem}
 \begin{cases}
  u_t=\Delta u - \chi \nabla \cdot \left(\frac uv\nabla v\right)\\
  v_t=\Delta v - v +u,  
 \end{cases}
\end{equation}
for recent studies of which we refer to \cite{fujiesenba2,lanwin} and the references therein.

Nevertheless, it is known that \eqref{intro:thesystem} admits global solutions if posed in $ℝ^2$ or $ℝ^3$, under a smallness condition on the initial data $(u_0,v_0)$, involving $H^2$-norms of $u_0$ and $\na v_0$, \cite{wangxiangyu_asymptotic}. In bounded, convex two-dimensional domains it is known from \cite{win_ct_sing_abs} that for arbitrarily large initial data, global solutions exist in a generalized sense. Moreover, they become eventually smooth if the initial mass $\io u_0$ is sufficiently small \cite{win_ct_sing_abs_eventual}. (The same article \cite{win_ct_sing_abs_eventual} also identifies a smallness condition on $(u_0,\na v_0)$ in $L\log L (\Om)\times L^2(\Om)$ which leads to global existence of classical solutions.) Similar results were achieved for a fluid-coupled variant of \eqref{intro:thesystem} in \cite{yilongwang} and \cite{black}. In the three-dimensional setting, however, the smallness condition in \cite{wangxiangyu_asymptotic} or, instead, restriction to the setting of radial 
symmetry and renormalized solutions \cite{win_ct_sing_abs_renormalized} seem to be necessary for all known proofs of global existence. 
 
Modifications that ensure global existence of solutions are using nonlinear diffusion of porous medium type (that is, replacing $Δu$ by $Δu^m$), which guarantees global existence (in bounded domains of $ℝ^n$) as long as $m>1+\f n4$, \cite{LankeitLocallyBoundedSingularity}, or weakening the cross-diffusive term by replacing $\f{u}{v} \na v$ by, essentially, an expression of the form $\f{u^{α}}v\na v$ for $α<1-\f{n}4$, \cite{dongmei}.  

Apart from these rather strong changes to the diffusive parts of the system, currently it seems that the two-dimensional case of \eqref{intro:thesystem} is just barely out of reach for global existence assertions concerning classical solutions emanating from rather general initial data, as witnessed by the fact that eventually smooth weak solutions exist (see above) or by recent results on how the presence of logistic source terms ($+κu-μu^2$ in the first equation) affects global solvability (\cite{lankeitlankeit1}): While in higher dimensions, global classical solvability results from $μ$ being sufficiently large, in bounded domains $\Om\subset ℝ^2$, any $μ>0$ suffices, provided that $χ<\sqrt{\frac2n}$, a number that also plays a role for global existence of solutions to \eqref{intro:productionsystem} (see \cite{fujie}). Furthermore, sources with stronger absorption, $+κu-μu^{α}$, $α>1+\f n2$, can ensure global existence, \cite{zhaozheng18}.

It can be expected that also lessening the impact of high values of the first solution component on the evolution of the second should enforce global existence of solutions. But by how much does it have to be lessened? What happens if the signal substance is consumed with a rate sublinearly depending on the bacterial density? 

Indeed, we will show that, at least for $χ<\sqrt{\f2n}=1$, any sublinear dependence of the consumption term on $u$ immediately suffices for globally existent classical solutions. Those will, moreover, remain bounded, if additionally the (initial) mass of bacteria is small. 
\section{Main result and structure of the paper}\label{StructureMainReusltSection}
In agreement with all of the above, this paper is dedicated to the following problem
\begin{equation}\label{problem}
\begin{cases}
u_t=\Delta u -\chi \nabla \cdot \big(\frac{u}{v} \nabla v\big)   & \text{in } \Omega\times(0,\infty), \\
v_t=\Delta v-f(u) v & \text{in } \Omega\times(0,\infty),\\
\frac{\partial u}{\partial \nu}=\frac{\partial v}{\partial \nu}=0 & \text{in } \partial\Omega\times(0,\infty),\\
u(x,0)=u_{0}(x) \quad \textrm{and}\quad v(x,0)=v_0(x),& x\in  \Omega,
\end{cases}
\end{equation}
defined in a bounded and smooth domain $\Omega$ of $\mathbb{R}^2$ and with $0<\chi<1$, 
where $f$ satisfies 
\begin{equation}\label{cond:f}
 f\in C^1(ℝ) \quad\text{and}\quad 0\le f(s)\le s^\beta \quad \text{for all } s>0
\end{equation}
and, occasionally, 
\begin{equation}\label{cond:fprime}
0\le f'(s)\le βs^{β-1} \qquad \text{for all }s> 0
\end{equation}
for some  $0<\beta<1$, and where 
\begin{equation}\label{initdata}
(u_0,v_0)\in C^0(\bar{\Omega})\times W^{1,r}(\Omega) \text{ for some } r>2, \text{ satisfy } u_0\ge 0 \text{ and } v_0>0 \text{ in  }\Ombar
\end{equation}
 are the initial distribution of cells and chemical concentration. Moreover, the zero-flux boundary conditions on both $u$ and $v$ model that the domain is totally insulated.

Under these assumptions we will show that classical solutions exist globally:
\begin{theorem}(Global existence)\label{MainTheorem}   
Let  
\begin{equation}\label{assumptions}\tag{A} 
 \begin{cases}
  \Omega\subset ℝ^2 \text{ be a smooth and bounded domain},\\
  \beta, \chi \in (0,1), \\
  f \text{ satisfy } \eqref{cond:f}.
 \end{cases}
\end{equation}
Then for any given $(u_0,v_0)$ as in \eqref{initdata},  there is a unique pair of functions $(u,v)$, 
\begin{align}\label{solutionregularity}
\begin{cases}
 u\in C^0(\bar{\Omega}\times [0,\infty))\cap C^{2,1}(\bar{\Omega}\times (0,\infty)),\\
 v\in C^0(\bar{\Omega}\times [0,\infty))\cap C^{2,1}(\bar{\Omega}\times (0,\infty))\cap L^\infty_{loc}([0,\infty), W^{1,r}(\Om)), \end{cases}
\end{align}
which solve problem \eqref{problem}.
\end{theorem}
Moreover, if additional smallness assumptions are imposed on the initial bacterial mass (an also biologically meaningful quantity), we can assert boundedness of these solutions.
\begin{theorem}(Boundedness)\label{MainTheoremBoundedness}
Let \eqref{assumptions} and \eqref{cond:fprime} be satisfied. 
Then, it is possible to find a positive $m_*$ with the property that for any given $(u_0,v_0)$ as in \eqref{initdata} 
and such that $\int_\Omega u_0(x)\leq m_*$, there is a unique pair of functions $(u,v)$ as in \eqref{solutionregularity} 
which solve problem \eqref{problem} and are bounded in $\Om\times(0,\infty)$.
\end{theorem}
The next section, Section \ref{ExistenceSolutionRegularizingSection}, will mainly be concerned with a local-in-time existence result. 
In particular, we show (in Lemma \ref{LemmaToEnsureTisInfty}) that controlling 
\[
 \norm[L^∞(\Omega)]{w(\cdot,t)}, \io u(\cdot,t){\log}u(\cdot,t), \text{ and } \int_0^t \io \f{|\na u |^2}{u}
\]
for $t\in [0,T)$, 
where $w:=-\log\kl{\f{v}{\norm[L^∞(\Omega)]{v_0}}}$ arises from the common transformation (see e.g. \cite{win_ct_sing_abs,LankeitLocallyBoundedSingularity}) which serves to replace $-\f{\nabla v}{v}$ by the nonsingular $\na w$, is sufficient for the conclusion that the solution exists longer than merely up to time $T$. 

In Section \ref{sec:globalexistence}, we then set out to derive bounds on these quantities in order to assert global existence of the previously found local solutions, and hence prove Theorem \ref{MainTheorem}. We will achieve this by consideration of the functional 
\[
 \io u\log u + a \io uw
\]
(whose usefulness in similar arguments pertaining to the different system \eqref{intro:productionsystem} has long been known, see \cite{Lankeit_newapproach,Biler99}, but which appears to be new for consumptive systems). 

The functional on whose properties the proof of eventual boundedness in \cite{win_ct_sing_abs_eventual} relies, is 
\[
 \io (u\log u+|\nabla w|^2); 
\]
we now (i.e. in Section \ref{sec:boundedness}, which is devoted to the proof of Theorem \ref{MainTheoremBoundedness}) treat the similar 
\[
 \calG(t)=\calG(u,w)=\f12\io |\na w|^2 + \io H(u), 
\]
see \eqref{GFunctionalBoundedness}, where $H$ is a second primitive of $σ\mapsto \f{f'(σ)}{χσ}$, in order to derive boundedness of $u$ on $[t_0,\infty)$ for some $t_0>0$, for small-mass solutions. Boundedness on finite time intervals $[0,t_0)$ is no longer an issue thanks to Theorem \ref{MainTheorem}.
% % % % % % % % % % % % % % % % % % % % % % % % % % % % % % % % % % % % % % % % % % % % % % % %
\section{Existence of local-in-time solutions and preparatory lemmas}\label{ExistenceSolutionRegularizingSection}
Let us firstly give a result concerning local-in-time existence of classical solutions to system \eqref{problem}. 
\begin{lemma}\label{LocalExistenceLemma}
Assume \eqref{assumptions}. 
Then, for  any given $(u_0,v_0)$ as in \eqref{initdata}, there are $T_{max}\in (0,∞]$ and a uniquely determined pair of functions $(u,v)$ with regularity as in \eqref{solutionregularity} which solve problem \eqref{problem} in $\Om\times(0,T_{max})$ and are such that if $T_{max}<\infty$ then
\begin{equation}\label{extensibility_criterion_Eq} 
\limsup_{t\nearrow T_{max}}(\lVert u (\cdot,t)\rVert_{L^\infty(\Omega)}+\lVert v (\cdot,t)\rVert_{W^{1,r}(\Omega)})=\infty.
\end{equation}
Moreover, we have 
\begin{equation}\label{Bound_of_u} 
\int_\Omega u (\cdot ,t)  =m=\int_\Omega u_0 \quad \textrm{for all}\quad  t\in (0,T_{max}),
\end{equation}
and 
\begin{equation}\label{BoundednessOfv}
u\geq 0 \quad \text{and}\quad  0<v\leq \lVert v_0 \rVert_{L^\infty(\Omega)}\quad \textrm{in}\quad \bar{\Omega} \times (0,T_{max}).
\end{equation}
\begin{proof}
The claim concerning the local existence and uniqueness as well as the extensibility criterion \eqref{extensibility_criterion_Eq} can be shown by straightforward
 adaptations of well-established methods involving an appropriate fixed point framework and standard parabolic regularity theory (see, for instance, %\cite{Cieslak}, \cite{HorstWink} and 
\cite{win_ct_sing_abs} or \cite{BellomoEtAl}).

On the other hand, taking into consideration the no-flux boundary conditions for problem \eqref{problem},  an integration of its first equation over $\Omega$ provides 
\begin{equation*} %\label{IntegrationOnOmegaFirstEq}
\frac{d}{dt}\int_\Omega u = 0\quad \textrm{for all}\quad t\in (0,T_{max}),
\end{equation*}
so that $\int_\Omega u=\int_\Omega u_0=m$ and \eqref{Bound_of_u} is shown.

Since $u_0\geq 0$ and $v_0>0$,  comparison arguments apply to yield both expressions in \eqref{BoundednessOfv}.
\end{proof}
\end{lemma}
Once the local existence of solutions to \eqref{problem} is attained, through the transformation 
\begin{equation}\label{transformationVtoW}
w:=-\log \Big(\frac{v}{\lVert v_0 \rVert_{L^\infty(\Omega)}}\Big), \qquad w_0:=-\log \Big(\frac{v_0}{\lVert v_0 \rVert_{L^\infty(\Omega)}}\Big), 
\end{equation}
which has already been used in \cite{LankeitLocallyBoundedSingularity} and \cite{win_ct_sing_abs}, 
we get that $w\geq 0$ in $\Omega \times (0,T_{max})$, and that $(u,w)\in  (C^0(\bar{\Omega}\times [0,T_{max}))\cap C^{2,1}(\bar{\Omega}\times (0,T_{max})))^2$ also (classically) solves the transformed problem
\begin{equation}\label{AuxiliaryMainsSystem}
\begin{cases}
u_t=\Delta u +\chi \nabla \cdot (u \nabla w) &  \text{ in }\Omega \times (0,T_{max}), \\
w_t=\Delta w -\rvert \nabla w \rvert^2+f(u) &  \text{ in }\Omega \times (0,T_{max}), \\
\frac{\partial u}{\partial \nu}=\frac{\partial w}{\partial \nu}=0 & \text{ in }\partial \Omega\times (0,T_{max}),\\
u(x,0)=u_{0}(x)\geq 0 \quad w(x,0)=w_0(x)\geq 0,& x\in  \bar{\Omega}.
\end{cases}
\end{equation}
It can be seen that  in this last system the first equation does not present the singularity at $v=0$ appearing in \eqref{problem}, so that this version will be considered in some places  in this paper.

Let us also recall those special cases of the well-known Gagliardo-Nirenberg inequality which will be used through the paper to prove the main theorems. 

\begin{lemma} (Gagliardo-Nirenberg inequality)\label{InequalityG-NLemma}
Let $\Omega$ be a bounded Lipschitz domain of $\R^2$. Then there is a constant $C_{GN}>0$ such that the following inequalities hold: 
With $\mathfrak{q},\mathfrak{s}\in\set{1,2}$, $\mathfrak{p}\in [2,4]$, $\theta=1-\f{\mathfrak{q}}{\mathfrak{p}}\in[0,1)$,  
\begin{equation}\label{InequalityTipoG-N} 
\| f \|_{L^{\mathfrak{p}}(\Omega)} \leq C_{GN} ( \| \nabla f \|_{L^{2}(\Omega)}^{\theta} \| f \|_{L^{\mathfrak{q}}(\Omega)}^{1 - \theta}+  \| f \|_{L^{\mathfrak{s}}(\Omega)})
\end{equation}
is satisfied for all $f\in L^\mathfrak{q}(\Omega)$ with $\nabla f\in L^2(\Omega)$,  
\begin{equation}\label{InequalityTipoG-N_Version2} 
\| f \|_{L^{\mathfrak{3}}(\Omega)} \leq C_{GN}  \|  f \|_{L^{\mathfrak{1}}(\Omega)}^{\f13} \| f \|_{W^{1,2}(\Omega)}^{\f23} \;\; \text{ for all } f\in W^{1,2}(\Omega).
\end{equation}
Finally, for any $f\in W^{2,2}(\Omega)$ fulfilling $\frac{\partial f}{\partial \nu}=0$ on $\partial  \Omega$ we have that 
\begin{equation}\label{LadyzhenskayaInequality} 
\| \nabla f \|_{L^{4}(\Omega)}^4 \leq C_{GN} \| \nabla f \|_{L^{2}(\Omega)}^{2} \| \Delta f \|_{L^{2}(\Omega)}^{2}.
\end{equation}
\begin{proof}
See \cite{Nirenber_GagNir_Ineque}.
\end{proof}
\end{lemma}

In order to avoid convexity conditions on the domain, let us recall the following estimate: 
\begin{lemma}\label{lem:bdry}
 Let $\Om\subset ℝ^n$, $n\geq 1$, be a bounded domain with smooth boundary. Then there is $C_{\partial\Om}>0$ such that for every function $f\in C^2(\Ombar)$ with $\f{\partial f}{\partial \nu}=0$ on $\partial\Om$, the inequality 
\[
 2\int_{\partial\Om} |\nabla f|^2 \f{\partial |\nabla f|^2}{\partial \nu} \le \f1{16} \io |\na|\na f|^2|^2 + C_{\partial\Om} \kl{\io |\na f|^2}^2
\]
 holds. 
\end{lemma}
\begin{proof}
A proof can be found in \cite[Prop. 3.2]{ishida_seki_yokota}. It is based on embeddings of the form $W^{r+\f12,2}(\Om)\hookrightarrow L^2(\partial\Om)$ for $r\in (0,\f12)$ and a Gagliardo--Nirenberg inequality for fractional Sobolev spaces, combined with estimates of $\f{\partial |\na w|^2}{\partial \nu}$ on $\partial \Om$. For convex domains, the left side actually is nonpositive.  
\end{proof}
The following result will enable us to estimate the spatio-temporal $L^2$-norm of the cells' density by their initial mass. 
\begin{lemma}\label{lem:estimate:L2:spacetime}
 Let $\Om\subset ℝ^2$ be a smooth and bounded domain, let $T>0$, $c_1,c_2>0$ and $m>0$. Then every function $u\in C^0(\Ombar\times[0,T))\cap C^{2,1}(\Ombar\times(0,T))$ which fulfils 
\[
 \io u(\cdot,t)=m \quad \text{and } \quad \int_0^t\io \f{|\na u|^2}{u } \le c_1t+c_2 \qquad \text{for all } t\in(0,T)
\]
for all $t\in(0,T)$ also satisfies 
\begin{equation}\label{GagliardoNirenberg_u^2_on(0,T)}
 \int_0^t \io u^2 \le C_1(m,c_1)t + C_2(m,c_2), 
\end{equation}
where 
\[
 C_1(m,c_1):=m(2C_{GN})^4(c_1+m),\quad C_2(m,c_2):=(2C_{GN})^4mc_2.
\]
\end{lemma}
\begin{proof}
The Gagliardo-Nirenberg inequality \eqref{InequalityTipoG-N} with % $n=2, \mathfrak{j}=0, \mathfrak{p}=4, \mathfrak{m}=1, \mathfrak{r}=\mathfrak{q}=\mathfrak{s}=2$ and $\theta=\frac{1}{2}$, together with 
  $\mathfrak{p}=4$, $\mathfrak{q}=\mathfrak{s}=2$ and $\theta=\frac{1}{2}$, together with 
\begin{equation}\label{AlgebraicInequality2toalpha}  
(A+B)^k\leq 2^k(A^k+B^k),
\end{equation}
valid for any $A,B\geq 0$ and $k>0$, enables us to estimate
 \begin{equation*} 
 \begin{split}
 \int_\Omega u^{2}&=\lVert \sqrt{u}\rVert^{4}_{L^{4}(\Omega)}\\ & \leq 
 [C_{GN}(\lVert \nabla \sqrt{u}\rVert^{\theta}_{L^{2}(\Omega)}\lVert \sqrt{u}\rVert^{(1-\theta)}_{L^{2}(\Omega)}+\lVert \sqrt{u}\rVert_{L^{2}(\Omega)})]^{4} \\ &
 \leq (2C_{GN})^{4}\Big[m\Big(\int_\Omega \frac{\lvert \nabla u \rvert^2}{u}\Big)+m^{2}\Big] \qquad \text{ on } (0,T),
 \end{split}
 \end{equation*} 
 and so for $t\in(0,T)$, thanks to the assumption on $\int_0^t\io \f{|\na u|^2}u$, we have that \eqref{GagliardoNirenberg_u^2_on(0,T)} holds.
\end{proof}
As a first application of Lemma \ref{lem:estimate:L2:spacetime}, let us sharpen the extensibility criterion \eqref{extensibility_criterion_Eq}. 
\begin{lemma}\label{LemmaToEnsureTisInfty} 
Let $\Omega$ be a smooth and bounded domain of $\mathbb{R}^2$ and $\chi\ge 0$, $\beta\in(0,1]$ 
and let $f$ be as in \eqref{cond:f}.
For any given $(u_0,v_0)$ as in \eqref{initdata}, let $(u,v)$ be the local-in-time classical solution of problem \eqref{problem} provided by Lemma \ref{LocalExistenceLemma}, and $(u,w)$ that of the transformed problem \eqref{AuxiliaryMainsSystem}, $w$ being the function introduced in \eqref{transformationVtoW}. If there exists a positive constant $C$ such that for all $t\in(0,T_{max})$ 
\begin{equation}\label{Bounds_of_SqrtU_And_w} 
\begin{cases}
\int_0^t\int_\Omega\f{|\na u|^2}{u } \leq C(1+t),\\
 \int_\Omega u(\cdot, t)\log u(\cdot,t)\leq C(1+t), \\
 \lVert w (\cdot,t)\rVert_{L^\infty(\Omega)}\leq C(1+t),
\end{cases}
\end{equation}
then $T_{max}=\infty.$
\begin{proof}
We will, to the contrary, assume $\Tmax$ finite,  and derive a contradiction to \eqref{extensibility_criterion_Eq}.
Let $m=\int_\Omega u=\int_\Omega u_0$, as in Lemma  \ref{LocalExistenceLemma}. Recalling \eqref{transformationVtoW}, we have that 
\begin{equation*}
\frac{1}{v}=\frac{e^w}{\lVert v_0\rVert_{L^\infty(\Omega)}},
\end{equation*}
and the third assumption in \eqref{Bounds_of_SqrtU_And_w} warrants that with some $L>0$
\begin{equation}\label{BoundFor1OverV}
 \frac1v\le L \qquad \text{in } \Om\times(0,\Tmax).
\end{equation}
 Now, from the estimate in Lemma \ref{lem:estimate:L2:spacetime}, we can deduce also some bound of $\int_\Omega \lvert \nabla v \rvert^2$ on $(0,\Tmax)$. In fact, a differentiation in time and the Young inequality allow us, through the second equation of \eqref{problem} and by estimate \eqref{cond:f}, to get 
 \begin{equation*}
 \begin{split}
 \frac{d}{dt}\int_\Omega \lvert \nabla v\rvert^2 &=2\int_\Omega \nabla v \cdot \nabla v_t=2 \int_\Omega \nabla v \cdot (\nabla \Delta v -\nabla (f(u) v)) %\\ &
 %=-2\int_\Omega (\Delta v)^2-2\int_\Omega \nabla v \cdot \nabla (f(u) v)
\\ &= -2\int_\Omega (\Delta v)^2+2\int_\Omega  (f(u) v)\Delta v\\ &
 \leq -2\int_\Omega (\Delta v)^2+\int_\Omega  (\Delta v)^2+\int_\Omega  u^{2\beta}v^2 \qquad \text{in } (0,\Tmax).
 \end{split}
 \end{equation*}
 Moreover, again from $β<1$, the Young inequality and \eqref{BoundednessOfv} we infer 
 \begin{equation*}
 \int_\Omega  u^{2\beta}v^2 \leq \lVert v_0 \rVert_{L^\infty(\Omega)}^2\beta \int_\Omega u^2 + \lVert v_0 \rVert_{L^\infty(\Omega)}^2(1-\beta)\lvert \Omega\rvert \qquad \text{in } (0,\Tmax),
 \end{equation*}
 so that, neglecting the nonpositive term  $-\int_\Omega (\Delta v)^2$ we obtain
 \begin{equation*}
 \frac{d}{dt}\int_\Omega \lvert \nabla v\rvert^2 \leq \lVert v_0 \rVert_{L^\infty(\Omega)}^2\beta \int_\Omega u^2 + \lVert v_0 \rVert_{L^\infty(\Omega)}^2(1-\beta)\lvert \Omega\rvert \qquad \text{in } (0,\Tmax).
 \end{equation*}
 Finally, by means of  \eqref{GagliardoNirenberg_u^2_on(0,T)}, an integration over $(0,t)$ yields, for 
 \[
 C_2=\lVert v_0 \rVert_{L^\infty(\Omega)}^2(\beta C_1 +(1-\beta)\lvert \Omega\rvert),
 \]
that
 \begin{equation}\label{BoundSquareNablav(0,T)}
 \int_\Omega \lvert \nabla v(\cdot,t)\rvert^2 \leq C_2(1+t)\quad \textrm{for all}\quad t\in (0,T_{max}).
 \end{equation}
Having derived these bounds, we will next attempt to reduce the present problem to the setting of the standard extensibility criterion of \cite[Sec. 3]{BellomoEtAl}.
 According to some ideas used in \cite{Lankeit_newapproach}, for the positive constant $L$ above introduced, let $\xi_L:\R \rightarrow [0,1]$ be a smooth, decreasing function verifying $\xi_L(v)=1$ for $v\leq 1/(2L)$ and $\xi_L(v)=0$ for $v\geq 1/L$. Subsequently, the function
 \begin{equation*}
 S(x,t,u,v)=
 \xi_L(v)\frac{2\chi}{L}+(1-\xi_L(v))\frac{\chi}{v}, \; \;\;\; (x,t,u,v) \in \bar{\Omega}\times [0,\infty)\times \R^2,
 \end{equation*}
 belongs to $C^{1+\omega}_{\textrm{loc}}(\bar{\Omega}\times [0,\infty)\times \mathbb{R}^2)$, for some $\omega \in (0,1)$, and, additionally, satisfies $S(x,t,u,v)\equiv \frac{\chi}{v}$ for all $v\geq 1/L.$

From all of the above and following the nomenclature of \cite[Sec. 3]{BellomoEtAl}, setting
\begin{align*}
%S(x,t,u,v)&=\frac{\chi}{v}, \text{if } v> \f1L,\;\,  \text{ and continued smoothly and boundedly for smaller values of } v \\
f(x,t,u,v_+)&\equiv 0,\quad \textrm{and}\quad g(x,t,u_+,v)=v-f(u) v,
\end{align*}
 the two partial differential equations of problem \eqref{problem} read 
 \[
 \begin{cases}
  u_t=\Delta u-\nabla \cdot (uS(x,t,u,v)\nabla v)+f(x,t,u,v)& \text{in }\Omega \times (0,T_{max})\\
  v_t=\Delta v-v+g(x,t,u,v) & \text{in }\Omega \times (0,T_{max}),
 \end{cases}
\]
by virtue of bound \eqref{BoundFor1OverV}.  
In particular,  besides   $S\in C^{1+\omega}_{\textrm{loc}}(\bar{\Omega}\times [0,\infty)\times \mathbb{R}^2)$, we have also that $f\in\ C^{1-}_{\textrm{loc}}(\bar{\Omega}\times [0,\infty)\times \mathbb{R}^2)$ and $g\in C^{1-}_{\textrm{loc}}(\bar{\Omega}\times [0,\infty)\times \mathbb{R}^2)$, as well as $f(x,t,0,v)=0$ for all $(x,t,v)\in  \bar{\Omega}\times [0,\infty)\times \mathbb{R}$ and $g(x,t,u,0 )=0$ for all $(x,t,u)\in \bar{\Omega}\times [0,\infty)\times \mathbb{R}$. 

After these preparations, we can conclude that $T_{max}=\infty$. Indeed, in view of the assumptions in \eqref{Bounds_of_SqrtU_And_w}, estimate \eqref{BoundSquareNablav(0,T)}, the regularity and boundedness of both $S$ and $v$ and the expression of $g$ given above, there exists a positive $N$ such that for all $t\in(0,T_{max})$ 
\begin{equation*} %\label{Bounds_of_SqrtU_And_w} 
\begin{cases}
\int_\Omega \lvert \nabla v(\cdot,t)\rvert^2\leq N,\quad \int_\Omega u(\cdot, t)\log u(\cdot,t)\leq N, \\ 
S(x,t,u,v)\leq N,\quad\lvert  g(x,t,u,v)\rvert \leq N(1+u).
\end{cases}
\end{equation*}
In such conditions, all the hypotheses  of \cite[Lemma 3.3]{BellomoEtAl} are accomplished and hence the same lemma implies boundedness of $t\mapsto\lVert u(\cdot, t)\rVert_{L^\infty(\Omega)}+\lVert v(\cdot, t)\rVert_{W^{1,r}(\Omega)}$  on $(0,T_{max})$. This contradicts the extensibility criterion \eqref{extensibility_criterion_Eq}, and therefore $T_{max}=\infty.$
\end{proof}
\end{lemma}
\section{Existence of global solutions: proof of Theorem \ref{MainTheorem}}\label{sec:globalexistence}
In this section, starting from the local solution $(u,w)$ to problem \eqref{AuxiliaryMainsSystem}, we define, for some proper $a>0$, whose precise value is to be chosen during the next proof, the following energy functional 
\begin{equation}\label{FuntionalUsedForGlobalExistence}
\mathfrak{F}=\mathfrak{F}(t)=\mathfrak{F}(u,w):=\int_\Omega u \log u+a\int_\Omega u w\quad \textrm{on}\quad (0,T_{max}),
\end{equation}
 and its initial value  
\[\mathfrak{F}(0)=\int_\Omega u _0\log u_0+a\int_\Omega u_0 w_0.\]
An investigation  of its time depending behaviour  will reveal useful estimates to be employed in the proof of Theorem \ref{MainTheorem}. 
\begin{lemma}\label{LemmaControllinguLoguAndnablauOveru}
Assume \eqref{assumptions}. 
Let $m>0$. Then there is $L_1=L_1(m)>0$ such that the following holds: For any given $(u_0,v_0)$ as in \eqref{initdata}, 
and additionally satisfying $\io u_0=m$, let $(u,v)$ be the local-in-time classical solution of problem \eqref{problem} provided by Lemma \ref{LocalExistenceLemma}, and $(u,w)$ that of the transformed problem \eqref{AuxiliaryMainsSystem}, $w$ being the function introduced in \eqref{transformationVtoW}. Then we can find a positive constant $L_{2}$ such that 
\begin{equation}\label{BoundOfEnergyFunctional}
\int_\Omega u (\cdot,t)\log u(\cdot,t) \leq L_1t+L_2 \quad\textrm{for all}\quad t \in (0,T_{max}),
\end{equation}
and
 \begin{equation}\label{FirstBoundNablaSqrtuOn(0,T)} 
 \int_0^t\int_\Omega \frac{\lvert \nabla u\rvert ^2}{u}\leq L_1t+L_2 \quad\textrm{for all}\quad t \in (0,T_{max}).
 \end{equation}
Moreover, $L_1(m)$ remains bounded in a neighbourhood  of $m=0$. 
\begin{proof}
With some positive $a \in (a_-,a_+)$, where $a_-$ and $a_+$ will be explicitly computed later,  starting from $\mathfrak{F}$ defined in \eqref{FuntionalUsedForGlobalExistence}, we firstly observe from the inequality $s \log s \geq -\frac{1}{e}$, valid for every $s> 0$, and nonnegativity of $u$ and $w$ that 
\begin{equation}\label{BoundedFromBelowF}
\mathfrak{F}(u,w)\geq -\frac{\lvert \Omega \rvert}{e}\quad \textrm{ on } (0,T_{max}).
\end{equation}
In view of \eqref{AuxiliaryMainsSystem}, a differentiation of $\mathfrak{F}$ and the divergence theorem provide
\begin{equation}\label{DerivativeF_1}
\begin{split}
\frac{d}{d t}\mathfrak{F}& =\int_\Omega (u_t \log u+u_t)+a\int_\Omega u_t w+a\int_\Omega u w_t\\ &
=\int_\Omega  \Delta u \log u +\chi \int_\Omega  \nabla\cdot (u \nabla w) \log u
+a\int_\Omega w \Delta u\\ &\quad +a\chi\int_\Omega w\nabla\cdot (u \nabla w)+a\int_\Omega u \Delta w - a\int_\Omega u \lvert \nabla w\rvert^2+a\int_\Omega uf(u)\\ &
\le-\int_\Omega \frac{\lvert \nabla u \rvert^2}{u}-(\chi+2a)\int_\Omega \nabla u \cdot \nabla w\\ &
\quad- a(\chi+1) \int_\Omega u \lvert \nabla w\rvert^2 +a\int_\Omega u^{\beta+1} \quad \textrm{on}\quad (0,T_{max}).
\end{split}
\end{equation}
Now, on the one hand the Young inequality implies
\begin{equation}\label{DerivativeF_2}
\begin{split}
-(\chi+2a)\int_\Omega \nabla u \cdot \nabla w&\leq a(\chi+1)\int_\Omega u \lvert \nabla w\rvert^2\\ &
\quad +\frac{(\chi+2a)^2}{4a(\chi+1)}\int_\Omega \frac{\lvert \nabla u\rvert^2}{u} \quad \text{on }(0,T_{max}),
\end{split}
\end{equation}
while on the other, again the Gagliardo-Nirenberg inequality \eqref{InequalityTipoG-N}, %with $n=2, \mathfrak{j}=0, \mathfrak{p}=2(\beta+1), \mathfrak{m}=1, \mathfrak{r}=\mathfrak{q}=\mathfrak{s}=2$ and $\theta=\frac{\beta}{\beta+1}$,
with $\mathfrak{p}=2(\beta+1)$, $\mathfrak{q}=\mathfrak{s}=2$ and $\theta=\frac{\beta}{\beta+1}$, and \eqref{AlgebraicInequality2toalpha}  control the term $a\int_\Omega u^{\beta+1}$ in this form: 
\begin{equation}\label{GagliardoNirenberBoundu^beta+1}
\begin{split}
\int_\Omega u^{\beta+1}&=\lVert \sqrt{u}\rVert^{2(\beta+1)}_{L^{2(\beta+1)}(\Omega)}\\ & \leq 
[C_{GN}(\lVert \nabla \sqrt{u}\rVert^{\theta}_{L^{2}(\Omega)}\lVert \sqrt{u}\rVert^{(1-\theta)}_{L^{2}(\Omega)}+\lVert \sqrt{u}\rVert_{L^{2}(\Omega)})]^{2(\beta+1)} \\ &
\leq (2C_{GN})^{2(\beta+1)}\Big[m\Big(\int_\Omega \frac{\lvert \nabla u \rvert^2}{u}\Big)^\beta+m^{\beta+1}\Big] \quad \textrm{on}\quad (0,T_{max}).
\end{split}
\end{equation} 
Now,  in view of the assumption $0<\chi<1$, the numbers
\[
a_-=\frac{1}{2} - \frac{1}{2}\sqrt{1 - \chi^2}\quad \textrm{and}\quad a_+=\frac{1}{2} + \frac{1}{2}\sqrt{1 - \chi^2}
\]
are real, and for any $a\in(a_-,a_+)$ the constant $c_0=1-\f{(χ+2a)^2}{4a(χ+1)}$ is positive. Hence an application of Young's inequality in \eqref{GagliardoNirenberBoundu^beta+1} shows that on $(0,\Tmax)$ 
\begin{equation}\label{GagliardoNirenberBoundu^beta+1Young}
\begin{split}
\int_\Omega u^{\beta+1}
\leq \frac{c_0}{2a}\int_\Omega \frac{\lvert \nabla u \rvert^2}{u}+c_1m^\frac{1}{1-\beta}+ (2C_{GN})^{2(\beta+1)} m^{\beta+1},
\end{split}
\end{equation} 
where $c_1=(1-\beta) (2C_{GN})^\frac{2(\beta+1)}{1-\beta}(\frac{2 a\beta }{c_0})^\frac{\beta}{1-\beta}$. By inserting \eqref{GagliardoNirenberBoundu^beta+1Young} and \eqref{DerivativeF_2} into \eqref{DerivativeF_1} and setting $c_2(m)=ac_1m^\frac{\beta}{1-\beta}+ a(2C_{GN})^{2(\beta+1)} m^{\beta+1}$ we get
\begin{equation}\label{InequalityForF}
\frac{d}{d t}\mathfrak{F} +\frac{c_0}{2}\int_\Omega \frac{\lvert \nabla u\rvert^2}{u}\leq c_2(m) \quad \textrm{on}\quad  (0,T_{max}),
\end{equation}
and an integration on $(0,t)$, for $t < T_{max}$, in conjunction with the bound from below of $\mathfrak{F}$ (expression \eqref{BoundedFromBelowF}), enables us to arrive at
\begin{equation*}
-\frac{\lvert \Omega \rvert}{e}\leq \mathfrak{F}(t) \leq \mathfrak{F}(0) +c_2(m) t \quad \textrm{for all}\quad t\in (0,T_{max}).
\end{equation*}
Thereupon,  in view of this bound, again an integration of \eqref{InequalityForF} on $(0,t)$, with $t<T_{max}$, leads to 
\[  
\frac{c_0}{2}\int_0^t\int_\Omega \frac{\lvert \nabla u\rvert^2}{u}\leq c_2(m) t+\mathfrak{F}(0)+\frac{\lvert \Omega \rvert}{e} \qquad \text{for } t\in(0,\Tmax),
\]
so that \eqref{BoundOfEnergyFunctional} and \eqref{FirstBoundNablaSqrtuOn(0,T)} are attained with the choices
\[
L_1:=L_1(m):=\max \Big\{c_2(m),\f{2c_2(m)}{c_0}\Big\}, \qquad L_2:=\f2{c_0}\kl{\mathfrak{F}(0)+\f{|\Om|}{e}}.
\]
Finally, since $c_0, c_1$ and $a$ do not depend on $m$, from the expression of $c_2(m)$ above we see that $L_1(m)$ is bounded close to $m=0.$ (More precisely, we even have that $ L_1(m) \searrow 0$ as $m\searrow 0$.)
\end{proof}
\end{lemma}
\begin{remark}
 Even though the functional $\mathfrak{F}$ has not been considered in \cite{win_ct_sing_abs_eventual} (nor in \cite{win_ct_sing_abs}), global existence of small-mass solutions in the case of $β=1$ (and $χ<1$) can be recovered from the above considerations: 
If we insert $β=1$ into \eqref{GagliardoNirenberBoundu^beta+1} and refrain from using Young's inequality, the remainder of the proof still is applicable, provided that $(2C_{GN})^4m<\f{c_0}{2a}$. 
\end{remark}

As a consequence of all of the above, we have the necessary ingredients to prove Theorem \ref{MainTheorem}.
\subsubsection*{Proof of Theorem \ref{MainTheorem}} 
Let $(u,w)$ be the local-in-time classical solution of problem \eqref{AuxiliaryMainsSystem} provided by Lemma \ref{LocalExistenceLemma}; clearly the function $w$ also solves
 \begin{equation*}
 w_t \leq \Delta w +f(u)  \qquad \text{in } \Omega \times (0,T_{max}),
 \end{equation*}
and therefore by using a representation formula and \eqref{cond:f} we get 
 \begin{equation*}
 w(\cdot,t)\leq e^{t \Delta}w_0+\int_0^t e^{(t-s)\Delta} u^\beta(\cdot,s) ds \quad \text{ in } \Om \text{ for any } t\in(0,\Tmax).
 \end{equation*}
 Now we invoke a standard estimate for the Neumann heat semigroup (see  \cite[Lemma 1.3]{WinklAggre}) which warrants the existence of a positive constant $C_S$  such that for all $t>0$ 
 \begin{equation*} %\label{LpLqEstimateGradient0}   
 \lVert  e^{t\Delta} φ  \lVert_{L^\infty(\Omega)}\leq C_S (1+t^{-\frac{1}{p}}) \lVert φ\lVert_{L^p(\Omega)}\quad \textrm{for all} \quad φ\in L^p(\Omega),
 \end{equation*}
 so that its application with $p=\frac{2}{\beta}$, in conjunction with the Young inequality, provides 
 \begin{equation}\label{LInftyBoundW}
\begin{split}
\lVert w (\cdot,t)\rVert_{L^\infty(\Omega)}& \leq \lVert e^{t \Delta}w_0 \rVert_{L^\infty(\Omega)}+\int_0^t  \lVert e^{(t-s)\Delta} u^{β} \rVert_{L^\infty(\Omega)}ds  \\ & 
\leq  \lVert w_0 \rVert_{L^\infty(\Omega)}+C_S\int_0^t (1+(t-s)^{-\frac{\beta}{2}})\bigg(\int_\Omega u^2\bigg)^\frac{\beta}{2}ds \\ &
 \leq  \lVert w_0 \rVert_{L^\infty(\Omega)}+\frac{C_S}{2}\int_0^t(1+(t-s)^{-\frac{\beta}{2}})^2 ds\\ &
 \quad +\frac{C_S}{2}\int_0^t \bigg(\int_\Omega u^2\bigg)^\beta ds \\ &
 \leq  \lVert w_0 \rVert_{L^\infty(\Omega)}+\frac{C_S}{2}t+\frac{C_S}{2(1-\beta)}t^{1-\beta}+\frac{2C_S}{2-\beta}t^\frac{2-\beta}{2}\\ &
 \quad 
 +\frac{C_S}{2}\beta \int_0^t\int_\Omega u^2+\frac{C_S}{2}(1-\beta)t,\quad t\in(0,T_{max}).
 \end{split}
 \end{equation}
According to Lemma \ref{lem:estimate:L2:spacetime} and Lemma \ref{LemmaControllinguLoguAndnablauOveru}, 
we can write 
 \begin{equation*}
\int_0^t\int_\Omega u^2\leq C_1(1+t)  \quad t<T_{max},
 \end{equation*}
 where 
 \begin{equation*} %\label{DefinitioConstantDependingOnMass_Bis}
C_1=C_1(m)=m (2C_{GN})^4(L_1(m)+m),
 \end{equation*}
with $L_1(m)>0$ from Lemma \ref{LemmaControllinguLoguAndnablauOveru}. Thereafter, \eqref{LInftyBoundW} becomes
 \begin{equation*} %\label{LInftyBoundW_Bis}
 \begin{split}
 \lVert w \rVert_{L^\infty(\Omega)}& \leq  \lVert w_0 \rVert_{L^\infty(\Omega)}+\frac{C_S}{2}t+\frac{C_S}{2(1-\beta)}t^{1-\beta}\\ &
  \quad +\frac{2C_S}{2-\beta}t^\frac{2-\beta}{2}
  +\frac{C_S}{2}\beta C_1(1+t)+\frac{C_S}{2}(1-\beta)
  \\ &
  \leq L_{3}(m)(1+t),\quad  t<T_{max},
  \end{split}
  \end{equation*}
where
\[
L_{3}(m):=\f{C_S}{2(1-β)}+\f{C_SC_1β}2+\f{2C_S}{2-β}+\max\bigg\{\norm[L^\infty(\Om)]{w_0}+\f{C_S}2(1-β),\f{C_S}2\bigg\}.
\]
Once the bounds \eqref{BoundOfEnergyFunctional},  \eqref{FirstBoundNablaSqrtuOn(0,T)}   and \eqref{LInftyBoundW} are considered, the conclusion is then a direct consequence of Lemma \ref{LemmaToEnsureTisInfty} with $C=\max\{L_1(m),L_2,L_3(m)\}$. 
 \qed
 \section{Deriving boundedness of global solutions: proof of Theorem \ref{MainTheoremBoundedness}}\label{sec:boundedness}
 Now that the global existence of solutions to \eqref{problem} is guaranteed, let us dedicate this section to the boundedness question: as we shall show in the sequel, this issue will be addressed if some smallness assumption on the initial mass $m$ is given.
 
 In particular, the boundedness of $u$ is achieved by controlling the quantity $\int_\Omega \lvert \nabla w\rvert^p$ for some $p>2$, as specified in this 
 \begin{lemma}\label{LemmaFromnablawpower4ToBoundedness}
 Assuming \eqref{assumptions}, let $m>0$, $p>2$, $K>0$ and $τ>0$. Then there is a positive constant $C=C(p,m,K,τ)$ such that for any initial data $(u_0,v_0)$ as in \eqref{initdata} with $\io u_0=m$, the solution $(u,w)$ provided by Theorem \ref{MainTheorem} and \eqref{transformationVtoW} satisfies the following: If 
 \begin{equation}\label{condonnablaw}
\int_\Omega \lvert \nabla w(\cdot,t)\rvert^p \leq K \quad \textrm{for all}\quad t>t_0,
 \end{equation}
 holds with some $t_0\ge 0$, then 
 \begin{equation}\label{givesboundednessofu}
 \lVert u(\cdot,t)\rVert_{L^\infty(\Omega)}\leq C \quad \textrm{for all}\quad t> t_0+τ.
 \end{equation}
 \begin{proof}
This is \cite[Lemma 4.4]{win_ct_sing_abs_eventual}, since only the first equation of the system is of importance here. We therefore only indicate the main steps, referring to \cite[Lemma 4.4]{win_ct_sing_abs_eventual} for details. 
Using that by Hölder's inequality and \eqref{condonnablaw} for $q\in(2,p)$ we have $\norm[\Lom q]{u(\cdot,s)\nabla w(\cdot,s)} \le m^{1-a} K^{\f1p} \norm[\Lom\infty]{u(\cdot,s)}^a$ with $a=1-\f{p-q}{pq}$, from semigroup estimates we can infer that with some $c_1, c_2>0$ 
\begin{align}\label{eq:boundednesssgestimate}
 \norm[\Lom\infty]{u(\cdot,t)} & \le c_1 m (t-t_0)^{-1} + c_2 m^{1-a} K^{\f1p} \int_{t_0}^t (t-s)^{-(\f12+\f1q)} \norm[\Lom\infty]{u(\cdot,s)}^a ds\nn\\
 &\le c_1 m (t-t_0)^{-1} + c_2 m^{1-a} K^{\f1p}c_3 S_1^a, \qquad t\in(t_0,t_0+1), 
\end{align}
where $S_1:=\max\set{(t-t_0)\norm[\Lom\infty]{u(\cdot,t)} \mid t\in[t_0,t_0+1]}$ and $c_3:=\int_0^1 (1-σ)^{-(\f12+\f1q)}σ^{-a} dσ<\infty$.
Multiplication of \eqref{eq:boundednesssgestimate} by $(t-t_0)$ shows that we can find an explicit expression of $c_4>0$ such that $S_1\le c_4$ (see also \cite[(4.31)]{win_ct_sing_abs_eventual}). For $T>t_0$ and $t\in[t_0+1,T)$ we similarly derive that (no matter whether $t-1<t_0+1$ or not) 
\begin{align*}
 \norm[\Lom\infty]{u(\cdot,t)} & \le c_1m+c_2m^{1-a} K^{\f1p} \int_{t-1}^t (t-s)^{-(\f12+\f1q)} \norm[\Lom\infty]{u(\cdot,s)}^a ds\\
 &\le c_1m + c_2m^{1-a} K^{\f1p}\left[c_3S_1^a + \f{2q}{q-2} S_2^a(T)\right], 
\end{align*}
where $S_2(T):=\max\set{\norm[\Lom\infty]{u(\cdot,t)}\mid t\in[t_0+1,T]}$.  In particular, with analogous arguments employed to derive the uniform bound for $S_1$, we can find $c_5>0$ such that $S_2(T)\leq c_5$. Finally, these estimates entail that $S(T):=\max\set{S_1,S_2(T)}$ is bounded, independently of $T$, proving \eqref{givesboundednessofu}, with, for instance, $C=\max\{c_4,\frac{c_4}{τ},c_5\}$.
\end{proof}

\end{lemma}
In view of this crucial result, our final aim is to provide conditions capable of justifying eventual bounds for $\nabla w$ in $L^p(\Omega)$, with some $p>2$. This will be achieved by means of the forthcoming derivations, most of them tied to properties of the functional
 \begin{equation}\label{GFunctionalBoundedness}
\mathcal{G}:= \mathcal{G}(t):=\mathcal{G}(u(\cdot,t),w(\cdot,t)) =\frac{1}{2}\int_\Omega \lvert \nabla w(\cdot,t)\rvert^2+ \int_\Omega H(u(\cdot,t)),
 \end{equation}
 defined for any $t>0$ and associated to the global classical solution $(u,w)$ of \eqref{AuxiliaryMainsSystem}, where 
\[
 H(ξ):=-\f1{χ} \int_0^{ξ}\int_s^{∞} \f{f'(σ)}{σ}dσ ds.
\]
Further, let us remark that $\mathcal{G}$ generalizes the functional employed in \cite{win_ct_sing_abs_eventual} and \cite{black} for $f(u)\equiv u$. 
 \begin{lemma}\label{LemmaEvolutionG}
 Assuming \eqref{assumptions}, \eqref{cond:fprime} and \eqref{initdata},
 let $(u,w)$ be the global classical solution of problem \eqref{AuxiliaryMainsSystem} provided by Theorem \ref{MainTheorem}  and \eqref{transformationVtoW}. Then 
 \begin{equation}\label{nabla_w^2BoundedByG} 
\int_\Omega \lvert \nabla w\rvert^2\leq 2\mathcal{G}(u,w) + \frac{2 m^\beta\lvert \Omega \rvert^{1-\beta} }{\chi(1-\beta)}\quad \text{on } (0,\infty)
 \end{equation}
 and
  \begin{equation}\label{G_UpperBound}
 \mathcal{G}(u(\cdot,t),w(\cdot,t)) \leq \frac{1}{2}\int_\Omega \lvert \nabla w (\cdot,t) \vert^2 \quad \textrm{for all}\quad t>0. 
  \end{equation}
  Moreover, we also have that 
    \begin{equation}\label{FirstTimeDeriivativeG}
    \begin{split}
\frac{d}{dt}\mathcal{G}(u,w)&+\f1{χ}\io \f{f'(u)|\na u|^2}{u}\\ & +\frac{1}{2}\Big(1-C_{GN}\int_\Omega \lvert \nabla w \rvert^2\Big) \int_\Omega (\Delta w)^2\leq 0\quad \textrm{on }\quad (0,\infty), 
    \end{split}
    \end{equation}
    where $C_{GN}$ is the constant introduced in Lemma  \ref{InequalityG-NLemma}.
 \begin{proof}
Due to \eqref{cond:fprime}, the function $H$ is nonpositive  on $(0,∞)$ and can be estimated by 
\[
 -H(ξ) \le \f1{χ} \int_0^{ξ}\int_s^{∞} βσ^{β-2} dσds = \f{1}{χ(1-β)}ξ^{β} \qquad \text{for all } ξ>0, 
\]
 so that Hölder's inequality implies 
 \begin{equation*}
 \begin{split}
 \frac{1}{2}\int_\Omega \lvert \nabla w\rvert^2=\mathcal{G}(u,w)- \int_\Omega H(u)\leq \mathcal{G}(u,w)+\frac{m^\beta \lvert \Omega  \rvert^{1-\beta}  }{\chi(1-\beta)},
  \end{split}
 \end{equation*}
 which warrants  \eqref{nabla_w^2BoundedByG}, whilst \eqref{G_UpperBound} is an easy consequence of the definition of $\mathcal{G}$ and nonpositivity of $H$. 
 
 Additionally, thanks to the first equation of \eqref{AuxiliaryMainsSystem} and the mass conservation property, i.e. $\int_\Omega u_t \equiv 0$ for all $t>0$, we have that 
 \begin{equation*}
 \begin{split}
 \frac{d}{dt}\int_\Omega H(u)&=\int_\Omega H'(u)(\Delta u +\chi \nabla \cdot (u\nabla w)) \\& 
=-\int_\Omega H''(u)\lvert \nabla u\rvert^2 - χ\int_\Omega H''(u)u\nabla u \cdot \nabla w\\
 &=-\f1{χ}\io \f{f'(u)|\na u|^2}u - \io f'(u)\na u\cdot \na w \quad \text{on } (0,∞).
\end{split}
 \end{equation*}
Moreover, from the second equation of \eqref{AuxiliaryMainsSystem}, the Young inequality and relation \eqref{LadyzhenskayaInequality}, we achieve on $(0,∞)$ that
 \begin{equation*}
  \begin{split}
  \frac{d}{dt}\frac{1}{2}\int_\Omega \lvert \nabla w\rvert^2&+\int_\Omega (\Delta w)^2=\int_\Omega \lvert \nabla w \rvert^2\Delta w+  \int_\Omega f'(u)\nabla u \cdot \nabla w \\& 
  \leq \frac{1}{2} \int_\Omega (\Delta w)^2+\frac{1}{2}\int_\Omega \lvert \nabla w\rvert^4 + \int_\Omega f'(u)\nabla u \cdot \nabla w 
  \\& 
    \leq \frac{1}{2} \int_\Omega (\Delta w)^2+\frac{ C_{GN}}{2}\int_\Omega \lvert \nabla w\rvert^2\int_\Omega (\Delta w)^2 +   \int_\Omega f'(u)\nabla u \cdot \nabla w, \\&  
 \end{split}
  \end{equation*}
  so that  by adding the latest two relations %  and by also using  the identity $\frac{4}{\beta^2}\lvert \nabla u^\frac{\beta}{2}\rvert^2=u^{\beta-2}\lvert \nabla u\rvert^2$ 
 we can conclude. 
 \end{proof}
 \end{lemma}
 The next result will be employed in the sequel to establish eventual boundedness of the term $\int_\Omega \lvert \nabla w\rvert^2$, such an estimate being strongly necessary  to our purposes.
 \begin{lemma}\label{LemmadefiningValueGTzer0}   
 Assuming \eqref{assumptions}, \eqref{cond:fprime} and \eqref{initdata},
let $(u,w)$ be the global classical solution of problem \eqref{AuxiliaryMainsSystem} provided by Theorem \ref{MainTheorem} and \eqref{transformationVtoW}. Moreover, let $C_{GN}$ be the constant from Lemma  \ref{InequalityG-NLemma}. If there exists $t_0\geq 0$ such that the functional $\mathcal{G}$ defined in \eqref{GFunctionalBoundedness} satisfies
 \begin{equation}\label{SmallnessAssumptionOnFunctionalGFort0}
\mathcal{G}(t_0)< \frac{1}{4 C_{GN} }-\frac{m^\beta\lvert \Omega \rvert^{\beta-1}}{\chi (1-\beta)},
 \end{equation}
  then 
  \begin{equation*} %\label{EventualDecreasingForG}
  \mathcal{G}'(t)\leq 0\quad \textrm{for all}\quad t>t_0.
  \end{equation*}
 \begin{proof}
As in \cite[ Lemma 3.4]{win_ct_sing_abs_eventual}, by taking into consideration  \eqref{nabla_w^2BoundedByG} and assumption \eqref{SmallnessAssumptionOnFunctionalGFort0} we see that 
\begin{equation*}
\begin{split}
\int_\Omega \lvert \nabla w(\cdot,t_0) \rvert^2\leq 2\mathcal{G}(t_0)+2 \frac{m^\beta\lvert \Omega \rvert^{\beta-1}}{\chi (1-\beta)}&< \frac{2}{4 C_{GN}}-2 \frac{m^\beta\lvert \Omega \rvert^{\beta-1}}{\chi (1-\beta)}+2 \frac{m^\beta\lvert \Omega \rvert^{\beta-1}}{\chi (1-\beta)}\\ & = \frac{1}{2 C_{GN}},
\end{split}
\end{equation*}
so that the set 
\begin{equation*}
S:=\left\{t\geq t_0 \;\bigl\rvert\;  C_{GN}\int_\Omega \lvert \nabla w(\cdot,τ) \rvert^2<\frac{1}{2}\quad \textrm{for all}\quad τ\in(t_0,t)\right\}
\end{equation*}
is not empty; more precisely, we aim to show that $T:=\sup S=\infty$. Indeed, if $T$ was finite, from the continuity of $t\mapsto \int_\Omega \lvert \nabla w(\cdot,t)\rvert^2$ we would necessarily have that 
\begin{equation}\label{ForContradictionAssumption}
C_{GN}\int_\Omega \lvert \nabla w(\cdot,T) \rvert^2=\frac{1}{2}.
\end{equation}
On the other hand, from \eqref{FirstTimeDeriivativeG} and in view of the nonnegativity of $f'$ due to \eqref{cond:fprime} it is inferred that 
    \begin{equation*}
    \begin{split}
\mathcal{G}'(t)\le -\f12\kl{1-C_{GN}\io |\na w(\cdot,t)|^2} \io |Δw(\cdot,t)|^2 \le -\f14 \io |Δw(\cdot,t)|^2 \le 0
%=\frac{d}{dt}\mathcal{G}(u,w)&\leq -\frac{4}{\chi\beta}\int_\Omega \lvert \nabla u^\frac{\beta}{2}\lvert^2-\frac{1}{4} \int_\Omega (\Delta w)^2\leq 0
%\quad \textrm{for all}\quad t\in [t_0,T),
    \end{split}
    \end{equation*}
for all $t\in[t_0,T)$, so that $\mathcal{G}(T)\le \mathcal{G}(t_0)$
    and, again by virtue of \eqref{nabla_w^2BoundedByG} and \eqref{SmallnessAssumptionOnFunctionalGFort0}, %an integration on $[t_0,T)$ of the last ordinary differential inequality yields 
    \begin{equation*}
\int_\Omega \lvert \nabla w(\cdot,T)\rvert^2\leq 2\mathcal{G}(T)+2 \frac{m^\beta\lvert \Omega \rvert^{\beta-1}}{\chi (1-\beta)}\leq  2 \mathcal{G}(t_0)+2 \frac{m^\beta\lvert \Omega \rvert^{\beta-1}}{\chi (1-\beta)}<\frac{1}{2 C_{GN}},
    \end{equation*}
which contradicts \eqref{ForContradictionAssumption}; then $T=\infty$ and the proof is given. 
 \end{proof}
 \end{lemma}
 \begin{lemma}\label{lemmaEstimateDerivativeusquare}  
Assuming \eqref{assumptions} and \eqref{initdata}, let $(u,w)$ be the global classical solution of problem \eqref{AuxiliaryMainsSystem} provided by  Theorem \ref{MainTheorem}  and \eqref{transformationVtoW}. Then for any positive $\varepsilon_1$ we have that  
\begin{equation}\label{timeDerivative_u^2}
\frac{d}{dt}\int_\Omega u^2 +\int_\Omega \lvert \nabla u\rvert^2 \leq \varepsilon_1\int_\Omega \lvert \nabla w\rvert^6+D_1(\varepsilon_1) \int_\Omega u^3 \quad \text{on } (0,∞)
\end{equation}
where $D_1(\varepsilon_1)=\frac{\chi^3}{3}(6\varepsilon_1)^{-\frac{1}{2}}$.
\begin{proof}
By multiplying the first equation of problem \eqref{AuxiliaryMainsSystem} by $u$, an integration by parts implies
\begin{equation}\label{thepreviousequality}
\frac{1}{2}\frac{d}{dt}\int_\Omega u^2+\int_\Omega \lvert \nabla u\rvert^2=-\chi \int_\Omega u \nabla u \cdot \nabla w \quad\text{on } (0,∞).
\end{equation}
The claim is obtained once we use in \eqref{thepreviousequality} that
\begin{equation*}
-\chi \int_\Omega u \nabla u \cdot \nabla w\leq \frac{1}{2}\int_\Omega \rvert \nabla u\rvert^2 + \varepsilon_1\int_\Omega \lvert \nabla w\rvert^6+D_1(\varepsilon_1) \int_\Omega u^3\quad \text{on } (0,\infty), 
\end{equation*}
achieved thanks to two applications of the Young inequality, the first with exponents $\frac{1}{2}$ and $\frac{1}{2}$ and the second with $\frac{1}{3}$ and $\frac{2}{3}.$
\end{proof}
 \end{lemma} 
The following results will all be aimed at controlling the size of $\int_\Omega \lvert \nabla w\rvert^2$ at large time $t$.

\begin{lemma}\label{lemmaEstimateDerivativeNablawFourth}  
 Assuming \eqref{assumptions} and \eqref{initdata}, let $(u,w)$ be the global classical solution of problem \eqref{AuxiliaryMainsSystem} provided by Theorem \ref{MainTheorem} and \eqref{transformationVtoW}. Then for any positive $\varepsilon_2$ we have that on $(0,∞)$
 \begin{equation}\label{timederivativenablaw^4}
 \begin{split}
 \frac{d}{dt}\int_\Omega \lvert \nabla w\rvert^4 +\f{9}{16}\int_\Omega\lvert \nabla  \lvert \nabla w\rvert^2\rvert^2& \leq \Big(\frac{16}{9}+96 \varepsilon_2\Big)\int_\Omega \lvert \nabla w\rvert^6+96 D_2(\varepsilon_2)\beta \int_\Omega u^3 \\ &
\quad + \Cdom \kl{\io |\na w|^2}^2 +96 D_2(\varepsilon_2)(1-\beta)\lvert \Omega \rvert,
  \end{split}
 \end{equation}
 where $D_2(\varepsilon_2)=\frac{2}{3}(3\varepsilon_2)^{-\frac{1}{2}}$ and $\Cdom$ is as in Lemma \ref{lem:bdry}.
 \begin{proof}
 %As to the addendum $ \frac{d}{d t}\int_\Omega| \nabla w |^4$, by 
By means of the identity $\Delta |\nabla w|^2=2 \nabla w \cdot \nabla \Delta w +2 |D^2 w|^2$ 
 %and $(|\nabla w|^4)_t=2 \nabla v \cdot \nabla \Delta v -2 |\nabla v|^2+2\nabla u \cdot \nabla v$, the latter obtained by differentiating and successively multiplying by $\nabla v$ the second equation of \eqref{problemRegularized}, 
 and using the second equation of \eqref{AuxiliaryMainsSystem} and its corresponding boundary conditions we can write 
% \begin{equation*}
% (|\nabla v|^2)_t=\Delta |\nabla v|^2 -2 |D^2 v|^2  -2 |\nabla v|^2+2\nabla u \cdot \nabla v.
% \end{equation*}
% Multiplying this gained relation by $|\nabla v|^2$ and integrating over $\Omega$ eventuate
 \begin{equation*} %\label{A_00}
 \begin{split}
& \frac{d}{dt}\int_\Omega  |\nabla w|^4 = 4 \int_\Omega \lvert \nabla w\rvert^2\nabla w \cdot (\nabla \Delta w -\nabla |\nabla w|^2+\nabla f(u))\\ &
=2\int_\Omega \rvert \nabla w\lvert^2\Delta \rvert \nabla w\lvert^2-4 \int_\Omega  |D^2w|^2|\nabla w|^{2}\\ &
\quad -4\int_\Omega \lvert \nabla w \rvert^2 \nabla w \cdot \nabla \lvert \nabla w\rvert^2
+4 \int_\Omega f'(u)\lvert \nabla w \rvert^2 \nabla u \cdot \nabla w  \quad \text{on } (0,∞).
 \end{split}
 \end{equation*}
Now %, due to the convexity of $\Omega$ and the boundary condition $\frac{\partial w}{\partial \nu}=0$, we have that $\frac{\partial |\nabla w|^2}{\partial \nu}\leq 0$ on $\partial \Omega$ (see \cite[Lemma 3.2]{TaoWinkParaPara}) so that 
integration by parts gives
 \begin{equation}\label{A_001}
 \begin{split}
 \frac{d}{dt}\int_\Omega  |\nabla w|^4 &+2 \int_\Omega \rvert\nabla  \lvert \nabla w\rvert^2\lvert^2+4\int_\Omega  |D^2w|^2|\nabla w|^{2} \\ & \leq  -4\int_\Omega \lvert \nabla w \rvert^2 \nabla w \cdot \nabla \lvert \nabla w\rvert^2 -4 \int_\Omega  f(u) \nabla w \cdot \nabla \lvert \nabla w \rvert^2 \\ &
\quad -4\int_\Omega f(u) \rvert \nabla w\lvert^2\Delta  w+2\intdom |\na w|^2 \f{\partial|\na w|^2}{\partial \nu} \quad \textrm{on }\quad (0,\infty),
 \end{split}
 \end{equation}
where, according to Lemma \ref{lem:bdry}, we can estimate 
\begin{equation}\label{est:bdry}
 2\intdom |\na w|^2 \f{\partial|\na w|^2}{\partial \nu} \le \f1{16} \io |\na |\na w|^2|^2 + \Cdom \kl{\io |\na w|^2}^2 \qquad \text{on } (0,\infty).
\end{equation}
 Subsequently the Young inequality produces
 \begin{equation}\label{A_01}
 \begin{split}
-4\int_\Omega \lvert \nabla w \rvert^2 \nabla w \cdot \nabla \lvert \nabla w\rvert^2  & \leq \frac{9}{4}\int_\Omega |\nabla |\nabla w|^2|^2+\frac{16}{9}\int_\Omega \lvert \nabla w \rvert^6 \quad \text{on } (0,\infty),
%\\ &\int_\Omega u \nabla v \cdot \nabla |\nabla v|^2-2\int_\Omega u |\nabla v|^2 \Delta v \\& 
%  \leq \frac{1}{2}\int_\Omega |\nabla |\nabla v|^2|^2+2 \int_\Omega u^2 |\nabla v|^2\\ &
%  \quad +\frac{2}{3}\int_\Omega |\nabla v|^2 |\Delta v|^2+\frac{3}{2}\int_\Omega u^2 |\nabla v|^2.
 \end{split}
 \end{equation}
 and
  \begin{equation}\label{A_02}
  \begin{split}
&-4 \int_\Omega f(u) \nabla w \cdot \nabla \lvert \nabla w \rvert^2 -4\int_\Omega f(u) \rvert \nabla w\lvert^2\Delta  w \\ &   \leq \frac{1}{16}\int_\Omega |\nabla |\nabla w|^2|^2 +64\int_\Omega u^{2\beta}\lvert \nabla w \rvert^2 +\frac{1}{8}\int_\Omega |\nabla w|^2|\Delta w|^{2} + 32\int_\Omega u^{2\beta}\lvert \nabla w \rvert^2\\ &
\leq \frac{1}{16}\int_\Omega |\nabla |\nabla w|^2|^2 +\frac{1}{4}\int_\Omega |D^2w|^2|\nabla  w|^{2}+96\int_\Omega u^{2\beta}\lvert \nabla w \rvert^2 \quad \textrm{on }\quad (0,\infty),
% \\ &\int_\Omega u \nabla v \cdot \nabla |\nabla v|^2-2\int_\Omega u |\nabla v|^2 \Delta v \\& 
%   \leq \frac{1}{2}\int_\Omega |\nabla |\nabla v|^2|^2+2 \int_\Omega u^2 |\nabla v|^2\\ &
%   \quad +\frac{2}{3}\int_\Omega |\nabla v|^2 |\Delta v|^2+\frac{3}{2}\int_\Omega u^2 |\nabla v|^2.
  \end{split}
  \end{equation}
  where we have used the pointwise relation $|\Delta w|^2 \leq 2 |D^2w|^2$, valid throughout $\Om$, and, of course, \eqref{cond:f}. 
  
  On the other hand, again two applications of the Young inequality for any $ε_2>0$ yield 
\begin{equation}\label{A_03}
\begin{split}
\int_\Omega u^{2\beta}\lvert \nabla w \rvert^2&\leq \varepsilon_2\int_\Omega \lvert \nabla w\rvert^6+D_2(\varepsilon_2)\int_\Omega u^{3\beta}
\leq \varepsilon_2\int_\Omega \lvert \nabla w\rvert^6+D_2(\varepsilon_2)\beta\int_\Omega u^{3}\\ &
\quad +D_2(\varepsilon_2)(1-\beta)\lvert \Omega\rvert \quad \text{on } (0,∞)
\end{split}
\end{equation}
with $D_2(\varepsilon_2)=\frac{2}{3}(3\varepsilon_2)^{-\frac{1}{2}}$.

Finally, by plugging \eqref{est:bdry}, \eqref{A_01}, \eqref{A_02} and \eqref{A_03} into \eqref{A_001}, and in view of the relation %(see Lemma 3.1 of \cite{ViglialoroMarrasMathemNacri})
\[\lvert \nabla \lvert \nabla w\rvert^2\rvert^2=4\lvert D^2w \nabla w \rvert^2\leq 4|D^2 w|^2\lvert \nabla w \rvert^2,\]
we readily have the claim.
\end{proof}
\end{lemma}
% % % % %
\begin{lemma}\label{LemmaWithNablawSqureBoundedAssumption} 
Assuming \eqref{assumptions} and \eqref{initdata}, 
let $(u,w)$ be the global classical solution of problem \eqref{AuxiliaryMainsSystem} provided by Theorem \ref{MainTheorem} and \eqref{transformationVtoW}.  If for some $M>0$ and $t_0\geq 0$
\begin{equation}\label{AssumptionOnnablaw^2Eventuall}
\int_\Omega \lvert \nabla w(\cdot,t)\rvert^2\leq M\quad \textrm{for all}\quad t >t_0,
\end{equation}
then we have that  
\begin{equation}\label{FirstTimeDerivativeuSquareNablawFourth}
\begin{split}
 \frac{d}{dt}\Big(\int_\Omega u^2&+\int_\Omega \lvert \nabla w\rvert^4\Big) +\kl{1-8 C^{3}_{GN}m(D_1(\varepsilon_1)+96βD_2(\varepsilon_2))}\int_\Omega \lvert \nabla u\rvert^2\\ &
 \quad +\kl{\f9{16}-2(\varepsilon_1+\frac{16}{9}+96\varepsilon_2)C_{GN}M}\int_\Omega\lvert \nabla  \lvert \nabla w\rvert^2\rvert^2   \\ &
\leq  \Cdom M^2 + 96 D_2(\varepsilon_2)(1-\beta)\lvert \Omega \rvert +8m^3 C^{3}_{GN}\kl{D_1(\varepsilon_1)+96βD_2(\varepsilon_2)}\\ &
\quad + \kl{\varepsilon_1+\frac{16}{9}+96\varepsilon_2}C^{2}_{GN}M^3 \quad \text{on } (t_0,\infty),
 \end{split}
 \end{equation}
where $C_{GN}$  is the constant introduced  in Lemma \ref{InequalityG-NLemma}, $\varepsilon_1$ and $\varepsilon_2$ are arbitrary positive constants and $D_1(\varepsilon_1)$ and $D_2(\varepsilon_2)$ have been defined in Lemmas  \ref{lemmaEstimateDerivativeusquare}   and \ref{lemmaEstimateDerivativeNablawFourth}.
\begin{proof}
By adding the inequalities \eqref{timeDerivative_u^2} and \eqref{timederivativenablaw^4}, both valid on $(0,∞)$, we get 
\begin{equation}\label{FirstTimeDerivative_usquare_nablavpower4} 
\begin{split}
 \frac{d}{dt}\Big(\int_\Omega u^2&+\int_\Omega \lvert \nabla w\rvert^4\Big) +\int_\Omega \lvert \nabla u\rvert^2+\f9{16} \int_\Omega\lvert \nabla  \lvert \nabla w\rvert^2\rvert^2   \\ &
\leq \Big(\varepsilon_1+\frac{16}{9}+96 \varepsilon_2\Big)\int_\Omega \lvert \nabla w\rvert^6 +(D_1(\varepsilon_1)+96β D_2(\varepsilon_2))\int_\Omega u^3\\ &
  \quad  \Cdom M^2 +96 D_2(\varepsilon_2)(1-\beta)\lvert \Omega \rvert \quad \text{on } (0,∞).
 \end{split}
 \end{equation}
% Now, for $n=2$, $\mathfrak{p}=3$ and $\mathfrak{q}=1$ we apply \eqref{InequalityTipoG-N_Version2} to achieve for $\theta=\frac{2}{3}$ and throughout $(0,∞)$ 
Now, we apply \eqref{InequalityTipoG-N_Version2} to achieve throughout $(0,∞)$ 
\begin{equation*}
\begin{split}
\int_\Omega \lvert \nabla w\rvert^6&=\|  \lvert \nabla w \rvert^2\|_{L^{3}(\Omega)}^3 \leq C_{GN}  \|   \lvert \nabla w \rvert^2 \|_{L^{1}(\Omega)} \|  \lvert \nabla w \rvert^2 \|_{W^{1,2}(\Omega)}^{2}\\ &
\leq C_{GN}\int_\Omega \lvert \nabla w\rvert^2\int_\Omega \lvert \nabla w\rvert^4+ C_{GN}\int_\Omega \lvert \nabla w\rvert^2\int_\Omega \lvert \nabla \lvert \nabla w\rvert^2\rvert^2,
\end{split}
\end{equation*}
and using 
\begin{equation*}
\int_\Omega \lvert \nabla w\rvert^4\leq \Big(\int_\Omega \lvert \nabla w\rvert^6\Big)^\frac{1}{2}\Big(\int_\Omega \lvert \nabla w\rvert^2\Big)^\frac{1}{2} \quad \text{on } (0,∞), 
\end{equation*}
we obtain through the Young inequality and assumption \eqref{AssumptionOnnablaw^2Eventuall}
\begin{equation}\label{Estim_For_nabla wpower6Final}    
\begin{split}
\int_\Omega \lvert \nabla w\rvert^6&
\leq C_{GN}\Big(\int_\Omega \lvert \nabla w\rvert^6\Big)^\frac{1}{2}\Big(\int_\Omega \lvert \nabla w\rvert^2\Big)^\frac{3}{2}+ C_{GN}\int_\Omega \lvert \nabla w\rvert^2\int_\Omega \lvert \nabla \lvert \nabla w\rvert^2\rvert^2 \\ &
\leq \frac{1}{2} \int_\Omega \lvert \nabla w\rvert^6+ \frac{C_{GN}^2}{2}\Big(\int_\Omega \lvert \nabla w\rvert^2\Big)^3 + C_{GN}\int_\Omega \lvert \nabla w\rvert^2\int_\Omega \lvert \nabla \lvert \nabla w\rvert^2\rvert^2  \\ &
\leq \frac{1}{2} \int_\Omega \lvert \nabla w\rvert^6+ \frac{C_{GN}^2M^3}{2}+ C_{GN}M\int_\Omega \lvert \nabla \lvert \nabla w\rvert^2\rvert^2 \quad \text{on } (t_0,∞). 
\end{split}
\end{equation}
% Turning our attention to the term $\int_\Omega u^3$, the Gagliardo-Nirenberg inequality \eqref{InequalityTipoG-N} with the particular choice $\mathfrak{j}=0$, $\mathfrak{p}=3$, $\mathfrak{m}=1$, $\mathfrak{s}=\mathfrak{q}=1$, $\mathfrak{r}=n=2$ and $\theta=\frac{2}{3}<1$, gives in conjunction with \eqref{AlgebraicInequality2toalpha},   
Turning our attention to the term $\int_\Omega u^3$, the Gagliardo-Nirenberg inequality \eqref{InequalityTipoG-N} with the particular choice $\mathfrak{p}=3$, $\mathfrak{s}=\mathfrak{q}=1$ and $\theta=\frac{2}{3}<1$, gives in conjunction with  \eqref{AlgebraicInequality2toalpha},   
 \begin{equation}\label{Estim_3_For_u^3}
 \begin{split}
 \int_\Omega u^{3}=\lvert \lvert u \lvert \lvert_{L^3(\Omega)}^3 
 &\leq \Big[C_{GN}\Big(        \lvert \lvert\nabla u\lvert \lvert_{L^2(\Omega)}^{\theta}      \lvert \lvert u\lvert \lvert_{L^1(\Omega)}^{1-\theta}  +     \lvert \lvert u\lvert \lvert_{L^1(\Omega)}\Big )\Big]^3\\&
 \leq (8C^{3}_{GN})\Big(     m \int_\Omega \lvert\nabla u  \lvert^2+m^3\Big )\quad \text{on } (0,∞),
 \end{split}
 \end{equation}
 where we also considered the mass conservation property, i.e. $\int_\Omega u=m.$  Finally, by using \eqref{Estim_For_nabla wpower6Final}  and \eqref{Estim_3_For_u^3}, \eqref{FirstTimeDerivative_usquare_nablavpower4}  reads exactly as in \eqref{FirstTimeDerivativeuSquareNablawFourth}.
\end{proof}
\end{lemma}
In the next result we shall show uniform-in-time boundedness of the $L^4(\Omega)$-norm of $\nabla w$ beyond some time, which will be used in order to obtain eventual boundedness of $u$. This is possible through a smallness assumption on $m$.
\begin{lemma}\label{LemmaFixingValueMofNbawSquare} 
Assume \eqref{assumptions}. For any $M\in (0,\frac{9}{17 \cdot 32 C_{GN}})$, it is possible to find $γ>0$ such that if a global solution $(u,w)$ of problem \eqref{AuxiliaryMainsSystem} emanates from initial data as in \eqref{initdata} and \eqref{transformationVtoW} and fulfilling $\int_\Omega u_0=:\bar{m}\leq \frac{1}{16 \gamma C_{GN}^3}$ and also satisfies \eqref{AssumptionOnnablaw^2Eventuall} of Lemma \ref{LemmaWithNablawSqureBoundedAssumption} for some $t_0>0$, 
then it has the following property: For any $τ>0$ there exists a positive constant $K$ such that 
\begin{equation*} 
\int_\Omega \lvert \nabla w(\cdot, t)\rvert^4\leq K \quad \textrm{for all}\quad  t>t_0+τ.
\end{equation*}
\begin{proof}
Let $D_1(\varepsilon_1)$ and $D_2(\varepsilon_2)$ be the $\varepsilon$-dependent functions defined in Lemmas  \ref{lemmaEstimateDerivativeusquare}  and \ref{lemmaEstimateDerivativeNablawFourth}. Since $M\in (0,\frac{9}{17 \cdot 32 C_{GN}})$, by choosing 
$$\varepsilon_1=\frac{1}{32 MC_{GN}}-\frac{17}{9}>0, \quad \varepsilon_2=\frac{1}{96\cdot 9}\quad \textrm{and}\quad \gamma=D_1(\varepsilon_1)+96\beta D_2(\varepsilon_2),$$
we have that, in view of the assumption $\int_\Omega u_0=\bar{m}\leq \frac{1}{16 \gamma C_{GN}^3}$, 
\[
2\kl{\varepsilon_1+\frac{16}{9}+96\varepsilon_2}C_{GN}M=\frac{1}{16}\quad \textrm{and}\quad 8 C^{3}_{GN}\bar{m}\gamma \leq \frac{1}{2}.
\]
Hence, through \eqref{AssumptionOnnablaw^2Eventuall} of Lemma \ref{LemmaWithNablawSqureBoundedAssumption}, inequality \eqref{FirstTimeDerivativeuSquareNablawFourth} reads
\begin{equation}\label{SecondTimeDerivative_usquare_nablavpower4}
\begin{split}
 \frac{d}{dt}\Big(\int_\Omega u^2&+\int_\Omega \lvert \nabla w\rvert^4\Big) +\frac{1}{2}\int_\Omega \lvert \nabla u\rvert^2
+ \frac{1}{2}\int_\Omega\lvert \nabla  \lvert \nabla w\rvert^2\rvert^2   
\leq c_4\quad \textrm{for all}\quad t>t_0,
 \end{split}
 \end{equation}
 with $c_4=\Cdom M^2+96 D_2(\varepsilon_2)(1-\beta)\lvert \Omega \rvert +8\bar{m}^3C_{GN}^3γ +C_{GN}M^2/16$. 
 
 Now we are in the favourable position to control the integrals involving $|\nabla u|^2$ and $\lvert \nabla  \lvert \nabla w\rvert^2\rvert^2$ by using again \eqref{InequalityTipoG-N} % for $n=2$, $\mathfrak{j}=0$, $\mathfrak{m}=1$, $\mathfrak{p}=\mathfrak{r}=2$, $\mathfrak{q}=\mathfrak{s}=1$ and $\theta=\frac{1}{2}$, which shows that 
 for $\mathfrak{p}=2$, $\mathfrak{q}=\mathfrak{s}=1$ and $\theta=\frac{1}{2}$, which shows that 
\[
 \norm[L^2]{f}^4 \le (2C_{GN})^4 \left(\norm[L^2]{\nabla f}^2\norm[L^1]{f}^2 + \norm[L^1]{f}^4\right)
\]
and hence for $f\in L^4(\Om)$ with $\nabla f\in L^2(\Om)$  
\begin{equation}\label{gni-otherform}
 -\f12 \norm[L^2]{\nabla f}^2 \le - \f{\norm[L^2]{f}^4}{2(2C_{GN})^4\norm[L^1]{f}^2} +\f{\norm[L^1]{f}^2}2.
\end{equation}
In particular, if we take into account the mass conservation property $\int_\Omega u=\bar{m}$, for all $t>0$ we can write 
 \begin{equation*}  
 \begin{split}
 -\frac{1}{2}\int_\Omega |\nabla u(\cdot,t)|^2&\leq -\frac{1}{2(2C_{GN})^4 \bar{m}^2}\left(\int_\Omega u^2(\cdot,t)\right)^{2}  +\frac{\bar{m}^2}{2}, %\\ &
% \leq-c_5\int_\Omega u^2  +c_6 m^2\quad \textrm{for all}\quad t>t_0.
 %{u^{p}}_{\Omega},\\
 %-\frac{q-1}{q^2}\int_\Omega|\nabla |\nabla  \ve|^q|^2\leq -\frac{q-1}{q^2C_P}\int_\Omega |\nabla \ve|^{2q}+|\Omega|\frac{q-1}{q^2C_P} [|\nabla \ve|^{q}]_{\Omega},
 \end{split}
 \end{equation*}
and similarly, by relying on \eqref{AssumptionOnnablaw^2Eventuall}, from \eqref{gni-otherform} we arrive at 
\begin{equation*}
  \begin{split}
  -\frac{1}{2}\int_\Omega \lvert \nabla |\nabla w(\cdot,t)|^2\rvert^2&\leq -\frac{1}{2 (2C_{GN})^4  M^2}\left(\int_\Omega |\nabla w(\cdot,t)|^4\right)^2 +\frac{M^2}{2}\quad \textrm{for all}\quad t>0.
  %{u^{p}}_{\Omega},\\
  %-\frac{q-1}{q^2}\int_\Omega|\nabla |\nabla  \ve|^q|^2\leq -\frac{q-1}{q^2C_P}\int_\Omega |\nabla \ve|^{2q}+|\Omega|\frac{q-1}{q^2C_P} [|\nabla \ve|^{q}]_{\Omega},
  \end{split}
  \end{equation*}
From these bounds, and setting $\Phi(t):=\int_\Omega u^2+\int_\Omega \lvert \nabla w\rvert^4$, the inequality \eqref{SecondTimeDerivative_usquare_nablavpower4} implies that $\Phi$ is a sub-solution of the ordinary differential equation 
\begin{equation}\label{AbsorptiveInequality}
%\begin{cases}
\Psi'(t)= -c_5\Psi^2(t)+c_6 \quad \textrm{for all}\quad \textrm{for all}\quad  t>t_0,\\
%\Phi(0)=\int_\Omega u_0^2+\int_\Omega \lvert \nabla w_0\rvert^4,
% \end{cases}
 \end{equation}
 with $c_5=\frac{1}{4(2C_{GN})^4 }\max\{\bar{m}^{2},M^{2}\}^{-2}$ and $c_6=c_4+\frac{M^2+\bar{m}^2}{2} $. 
 
By considering the function $\bar{\Phi}(t):=\frac{1}{c_5(t-t_0)}+\sqrt{\frac{c_6}{c_5}}$, $t\in(t_0,∞)$, we see that 
\[
\bar{\Phi}'(t)+c_5\bar{\Phi}^2(t)-c_6=
2\sqrt{\f{c_6}{c_5}}(t-t_0)^{-1}
%\frac{2+c_5 \left(4 \sqrt{2} \sqrt{\frac{c_6}{c_5}}+c_6 (t-t_0)\right)(t-t_0)}{c_5 (t-t_0)^2}
\geq 0 \quad \textrm{for all}\quad  t>t_0,\]
so that $\bar{\Phi}$ is a super-solution of \eqref{AbsorptiveInequality} such that $\bar{\Phi}(t) \nearrow +\infty$ as $t\searrow t_0$. Subsequently an ODE comparison reasoning leads to $\Phi(t) \leq \bar{\Phi}(t)$ for all $t>t_0$ and in particular we have
\[
\io |\na w(\cdot,t)|^4\le \Phi(t)\leq \bar{\Phi}(t_0+τ)=\frac{1}{c_5τ}+\sqrt{\frac{c_6}{c_5}}\quad \textrm{for all} \quad t\geq t_0+τ. \qedhere 
\]
 \end{proof}
 \end{lemma}
 \begin{lemma}\label{LemmaControllingNablawSquare}
 Under the assumption \eqref{assumptions}, for any $\varGamma>0$, it is possible to find $\hat{m}(\varGamma)>0$ such that for any initial data $(u_0,v_0)$ as in \eqref{initdata} and also fulfilling $\int_\Omega u_0\leq \hat{m}$, there is $t_*> 0$ such that the corresponding global classical solution $(u,w)$ of problem \eqref{AuxiliaryMainsSystem} 
satisfies
  \begin{equation*}
  \int_\Omega \lvert \nabla w(\cdot,t_*)\rvert^2 \leq \varGamma.
  \end{equation*}
 \begin{proof}
Given $m>0$, we let $L_1(m)$ be as in Lemma \ref{LemmaControllinguLoguAndnablauOveru} and $C_1(m):=C_1(m,L_1(m))$ the corresponding constant introduced in Lemma \ref{lem:estimate:L2:spacetime}. Since $L_1(m)$ remains bounded  in a  neighbourhood  of $m=0$, $C_1(m)\searrow 0$ as $m\searrow 0$, so, corresponding to $\varGamma>0$, we then choose $\hat{m}=\hat{m}(\varGamma)>0$ such that $C_1(\hat{m})<2^{-1-\f2{β}}|\Om|^{1-\f2{β}} \varGamma^{\f2{β}}$. 
From integrating the second equation of \eqref{AuxiliaryMainsSystem} over $(0,t)\times \Om$, \eqref{cond:f} and Hölder's inequality, we obtain for any $t>0$ that 
 \begin{equation}\label{Bla}
 \int_\Omega w -\int_\Omega w_0+\int_0^t \int_\Omega \lvert \nabla w\rvert^2= \int_0^t\int_\Omega f(u)\leq (t|\Omega|)^{1-\frac{\beta}{2}}\Big(\int_0^t\int_\Omega u^2\Big)^\frac{\beta}{2}.
 \end{equation}
Owing to the estimate $\int_0^t\io u^2\le C_1t+C_2$ of Lemmata \ref{LemmaControllinguLoguAndnablauOveru} and \ref{lem:estimate:L2:spacetime} combined, where $C_2:=C_2(m,L_2(u_0,w_0))$, due to the nonnegativity of $w$ 
% relations \eqref{AlgebraicInequality2toalpha}  and \eqref{GagliardoNirenberg_u^2_on(0,T)}
 we deduce from \eqref{Bla} and \eqref{AlgebraicInequality2toalpha}  that
 \begin{equation*}
  \begin{split}
 \int_0^t \int_\Omega \lvert \nabla w\rvert^2\leq  |\Omega|^{1-\frac{\beta}{2}}t^{1-\frac{\beta}{2}}\kl{(2C_1t)^{\f{β}2}+(2C_2)^{\f{β}2}}+\int_\Omega w_0 \quad \textrm{for all} \quad t>0,
  \end{split}
  \end{equation*}
 and for any $t>0$ the average theorem establishes the existence of a time $t_*\in(\frac{t}{2},t)$ such that 
  \begin{equation*} %\label{Bla2}
   \begin{split}
 \int_\Omega \lvert \nabla w(\cdot,t_*)\rvert^2&=\frac{2}{t} \int_{\frac{t}{2}}^t \int_\Omega \lvert \nabla w\rvert^2\leq \frac{2}{t} \int_0^t \int_\Omega \lvert \nabla w\rvert^2\\ & \leq 
2^{1+\f{β}2}|\Om|^{1-\f{β}2} C_1^{\f{β}2} + 2^{\f{β}2+1}|\Om|^{1-\f{β}2}C_2^{\f{β}2}t^{-\f{β}2} + \f2t \io w_0. 
%\frac{2C_2\hat{C}^\frac{\beta}{2}}{t^\frac{\beta}{2}}+2C_2\hat{C}^\frac{\beta}{2}+\frac{2\int_\Omega w_0}{t},
   \end{split}
   \end{equation*}
According to our choice of $\hat{m}$, for any initial data with $\io u_0\le \hat{m}$ it  is therefore apparently possible to choose $t$ sufficiently large so as to conclude the existence of $t_*$ satisfying 
% 
%    with $C_2= |\Omega|^{1-\frac{\beta}{2}}2^\frac{\beta}{2}.$ Finally, by taking $\hat{m}$ so that $\hat{C}=C_1(\hat{m})\leq (\frac{\varGamma}{6 C_2})^\frac{2}{\beta}$, and 
%  \[ t\geq \max\bigg\{1,\frac{6\int_\Omega w_0}{\varGamma}\bigg\},\] 
%  from the previous inequality we have that there is $t_*$ such that 
    \begin{equation*}
     \begin{split}
   \int_\Omega \lvert \nabla w(\cdot,t_*)\rvert^2\leq\varGamma,
     \end{split}
     \end{equation*}
     and the proof is concluded.
 \end{proof}
 \end{lemma}
 With the above derived information, and assuming a suitable smallness condition on $m=\int_\Omega u\equiv \int_\Omega u_0$, we can now ensure eventual boundedness of the spatial $L^2$-norm of $\nabla w$. Precisely we have

 \begin{lemma}\label{LemmaEventualBoundednessnablaw}
 Assume \eqref{assumptions} and \eqref{cond:fprime}. 
%Under the  assumptions  of Theorem \ref{MainTheorem}, let $(u,w)$ be the global classical solution of problem \eqref{AuxiliaryMainsSystem} provided by Theorem \ref{MainTheorem} and  \eqref{transformationVtoW}. 
For any $M>0$, there is $m_*>0$ such that
%There exist positive constants $\varGamma,$ $m_*$,  and $M$ such that 
%      \begin{equation*}
%      0< \varGamma < \frac{1}{5 C_{GN} }-\frac{m_*^\beta\lvert \Omega \rvert^{\beta-1}}{\chi (1-\beta)} \quad \textrm{and}\quad M<\frac{9}{17\cdot 16 C_{GN}},
%       \end{equation*}
%   and with the property that
 for any initial data $(u_0,v_0)$  as in \eqref{initdata} and also fulfilling   $\int_\Omega u_0\leq m_*$, there is $t_*> 0$ such that the corresponding global classical solution $(u,w)$ of problem \eqref{AuxiliaryMainsSystem} provided by Theorem \ref{MainTheorem} and \eqref{transformationVtoW} 
satisfies 
     \begin{equation}\label{EventualBoundedessnablaw}
   \int_\Omega \lvert \nabla w(\cdot,t)\rvert^2 \leq M \quad \text{for all } \quad t \geq t_*.
     \end{equation}
 \begin{proof}
Without loss of generality, we assume $M\in (0,\frac{9}{17\cdot 32 C_{GN}})$ and with 
 $\bar{m}$ from Lemma \ref{LemmaFixingValueMofNbawSquare}, let us set 
\begin{equation}\label{DefinitionmStar}
m   <  \min\Big\{\Big(\frac{\chi   (1-\beta)}{4\lvert \Omega \rvert^{\beta-1}C_{GN}}\Big)^\frac{1}{\beta},\Big(\frac{M\chi(1-\beta)}{4\lvert \Omega \rvert^{1-\beta}}\Big)^\frac{1}{\beta},\bar{m}\Big\},
\end{equation}
which in  particular implies 
\begin{equation*}
\frac{1}{4 C_{GN}}- \frac{m^\beta \lvert \Omega \rvert^{\beta-1}}{\chi(1-\beta)}>0.
\end{equation*}
Now let us pick 
\begin{equation}\label{DefiVarGamma}
0<\varGamma < \min\Big\{\frac{2}{4 C_{GN}}- \frac{2m^\beta\lvert \Omega \rvert^{\beta-1}}{\chi(1-\beta)},\frac{M}{2}\Big\}.
\end{equation}
In light of Lemma \ref{LemmaControllingNablawSquare},  we can find $\hat{m}=\hat{m}(\varGamma)$ so that for any solution emanating from initial data with $\io u_0<\hat{m}$ there is $t_*>0$ so that
% 
% 
% choose $t^*$ large enough so to warrant (recall that from its definition $m_*\leq \hat{m}$) that 
\begin{equation}\label{BoundNablawSquareFor FirnalTheroem}
%\begin{cases}
%m_*   < \max\Big\{\big(\frac{\chi \lvert \Omega \rvert^{\beta-1}  (1-\beta)}{6C_{GN}}\big)^\frac{1}{\beta},\Big(\frac{M\chi(1-\beta)}{4\lvert \Omega %\rvert^{1-\beta}}\Big)^\frac{1}{\beta}\Big\}, \\
%0<\varGamma < \frac{1}{5 C_{GN} }-\frac{m{_*}^\beta\lvert \Omega \rvert^{\beta-1}}{\chi (1-\beta)},\\
 \int_\Omega \lvert \nabla w(\cdot,t_*)\rvert^2 \leq \varGamma.
%\end{cases}
\end{equation}
Letting $m_*:=\min\set{\hat{m},m}$ and assuming that $\io u_0\le m_*$ and $t_*$ is such that \eqref{BoundNablawSquareFor FirnalTheroem} holds, thanks to  \eqref{G_UpperBound} we have that
\begin{equation*}
\mathcal{G}(t_*)\leq \frac{1}{2}\int_\Omega  \lvert \nabla w (\cdot, t_*)\lvert^2 \leq \frac{\varGamma}{2}<\frac{1}{4 C_{GN}}- \frac{m_*^{β} \lvert \Omega \rvert^{\beta-1}}{\chi(1-\beta)}.
\end{equation*}
We are now in the position to apply Lemma \ref{LemmadefiningValueGTzer0} and conclude that $\mathcal{G}'(t)\leq 0$ for all $t>t_*$, which subsequently provides that  $\mathcal{G}(t)\leq \mathcal{G}(t_*)$ for all $t\geq t_*$. Thereafter, from \eqref{nabla_w^2BoundedByG} of Lemma \ref{LemmaEvolutionG} we have that
 \begin{equation*}
\int_\Omega \lvert \nabla w(\cdot,t)\rvert^2\leq 2\mathcal{G}(t) + \frac{2 m_*^\beta\lvert \Omega \rvert^{1-\beta} }{\chi(1-\beta)}\quad \textrm{for all}\quad t>0,
 \end{equation*}
 i.e. through \eqref{G_UpperBound} and  \eqref{DefinitionmStar}-\eqref{BoundNablawSquareFor FirnalTheroem}
  \begin{equation*}
  \begin{split}
 \int_\Omega \lvert \nabla w(\cdot,t)\rvert^2& \leq 2\mathcal{G}(t_*) + \frac{2 m_*^\beta\lvert \Omega \rvert^{1-\beta} }{\chi(1-\beta)}\\ & \leq  \int_\Omega \lvert \nabla w(\cdot,t_*)\rvert^2+\frac{2 m_*^\beta\lvert \Omega \rvert^{1-\beta} }{\chi(1-\beta)}\leq M \quad \textrm{for all}\quad t\geq t_*.\qedhere
 \end{split}
  \end{equation*}
 \end{proof}
 \end{lemma}
 \subsubsection*{Proof of Theorem \ref{MainTheoremBoundedness}} 
Let $m_*>0$ be the value introduced in Lemma \ref{LemmaEventualBoundednessnablaw} and let $(u,w)$ be the global classical solution of problem \eqref{problem} provided by Theorem  \ref{MainTheorem} and \eqref{transformationVtoW}, and emanating from  initial data $(u_0,v_0)$  as in \eqref{initdata} and such that $\int_\Omega u_0\leq m_*$. By virtue of Lemma  \ref{LemmaEventualBoundednessnablaw}, we can find $t_*\in(0,\infty)$ such that relation \eqref{EventualBoundedessnablaw} holds with $M\in(0,\frac{9}{17\cdot 32 C_{GN}})$ and Lemma \ref{LemmaFixingValueMofNbawSquare} becomes applicable so that with some $K>0$ 
 \begin{equation*}  
 \int_\Omega \lvert \nabla w(\cdot, t)\rvert^4\leq K \quad \textrm{for all}\quad  t>t_*+1.
 \end{equation*}
Finally Lemma \ref{LemmaFromnablawpower4ToBoundedness} with the choice $p=4$, $t_0=t_*+1$ and $τ=1$ provides the boundedness of $u$ in $(t_*+2,\infty).$ Due to continuity and hence boundedness of $u$ in $\Ombar\times[0,t_*+2]$, this concludes the proof.
 \qed

\subsubsection*{Acknowledgments}
GV is member of the Gruppo Nazionale per l'Analisi Matematica, la Probabilit\`a e le loro Applicazioni (GNAMPA) of the Istituto Na\-zio\-na\-le di Alta Matematica (INdAM) and gratefully acknowledges the Italian Ministry of Education, University and Research (MIUR) for the financial support of Scientific Project ``Smart Cities and Communities and Social Innovation - ILEARNTV  anywhere, anytime - SCN$\_$00307''.
This work was initiated while JL was visiting the Universit\`{a} di Cagliari, in the framework of the project ``Blow-up and global existence of solutions to Keller-Segel type systems modelling chemotaxis'' (INdAM-GNAMPA  Project 2016). He is grateful for the hospitality.

% 
% \bibliography{Bibliography}{}
% %\bibliography{reference}
% \bibliographystyle{abbrv}
%\bibliographystyle{AIMS} 
%\bibliographystyle{plainnat}

\end{document}